\documentclass{amsart}

\usepackage{amsthm, amssymb}

\usepackage[norule,symbol,perpage]{footmisc}

\usepackage[colorlinks=true, linkcolor=blue]{hyperref}

\newcommand{\mc}{\mathcal}

\newcommand{\sub}{\subseteq}

\newcommand{\ol}{\overline}

\newcommand{\lra}{\Leftrightarrow}

\newcommand{\ra}{\Rightarrow}

\newcommand{\sm}{\setminus}

\DeclareMathOperator{\height}{height}

\DeclareMathOperator{\trd}{tr. d.}

\DeclareMathOperator{\vv}{\mathbf v}

\DeclareMathOperator{\Frac}{Frac}

\DeclareMathOperator{\U}{\mc U}

\newtheorem{theorem}{Theorem}[section]

\newtheorem{question}[theorem]{Question}

\newtheorem{lemma}[theorem]{Lemma}

\newtheorem{proposition}[theorem]{Proposition}

\newtheorem{corollary}[theorem]{Corollary}

\newtheorem{remark}[theorem]{Remark}

\newtheorem*{remarks*}{Remarks}

\newtheorem{example}[theorem]{Example}

\newtheorem{construction}[theorem]{Construction}

\theoremstyle{definition}

\newtheorem{notation}[theorem]{Notation}

\author{Stefania Gabelli}
\address{Dipartimento di Matematica e Fisica, Universit\`{a} degli Studi Roma
Tre,
Largo S.  L.  Murialdo,
1, 00146 Roma, Italy
}
\email{gabelli@mat.uniroma3.it}

\author{Moshe Roitman}

\address{Department of Mathematics, University of Haifa,
99 Abba Khoushy Avenue,
Mount Carme, Haifa 3498838, Israel}
\email{mroitman@math.haifa.ac.il}

\date{}

\newcommand{\comm}{,}
\newcounter{npart}
\renewcommand{\thenpart}{\Roman{npart}}
1
\title[On Finitely  Stable Domains\comm\ \thenpart]{}
\begin{document}
\begin{center}
{\LARGE \textbf{\color{red}{On Finitely Stable Domains}}}\\	
Stefania Gabelli and Moshe Roitman	
\end{center}	
\tableofcontents
\bigskip

{\it Part I of this paper corresponds to \cite{GR1}, and Part II to \cite{GR}.
\vspace*{\fill}

\newpage

\maketitle
\part{On Finitely Stable Domains, I}

\begin{abstract}
We prove that an integral domain $R$ is stable and one-dimensional if and only if $R$ is finitely stable and Mori.
If $R$ satisfies these two equivalent conditions, then each overring of $R$ also satisfies these conditions and it is $2$-$v$-generated.
We also prove that if $R$ is an Archimedean stable domain such that $R^\prime$ is local, then $R$ is one-dimensional and so Mori.
\end{abstract}

}\section{Introduction}
In this introduction we start with a short remainder of finitely stable and stable rings, recall the definitions of other classes of rings that we use here, as  Mori, Archimedean, etc., and finally summarize our main results. By a ring we mean a commutative ring with unity. A \emph{local ring} is a ring with a unique maximal ideal, not necessarily Noetherian. A {\em semilocal ring} is a ring with just finitely many  maximal ideals.  

Motivated by earlier work of H. Bass \cite{b} and J. Lipman \cite{L} on the number of generators of an
ideal, in 1972 J. Sally and W. Vasconcelos defined an ideal $I$ of a  ring $R$ to be \emph{stable} if $I$ is projective over its endomorphism ring; they called $R$ a \emph{stable ring} if  each nonzero ideal of $R$ is stable \cite{sv1,sv2}. Stability of rings is often determined by the stability of regular ideals, that is, ideals containing a nonzero divisor.
D. Rush studied the rings such that  each finitely generated regular ideal is stable, in particular in connection with properties of their integral closure and to the 2-generator property \cite{R1, R2}. These rings are now called \emph{finitely stable}.

In a note of 1987, D.D. Anderson, J. Huckaba and
I. Papick  considered the notion of stability for
integral domains \cite{AHP}. If $I$ is a nonzero ideal of a domain $R$,
then the endomorphism ring of $I$ coincides with the overring $E(I)=(I:I)$ of $R$; also, $I$ is projective over $E(I)$ if and only if $I$ is invertible as an ideal of $E(I)$. We use here notations like $(I:I)$ in a more general context: If $R$ and $T$ are domains with the same field  of fractions $K$, $I$ is an ideal of $R$ and $S$ is a subset of $K$, we set  $(I:_TS)=\{t\in T\ |\ tS \subseteq I \}$ and $(I:S)=(I:_K S)$. The stability property of a nonzero ideal $I$ does not depend on the domain containing $I$: more precisely, if $I$ is a common nonzero ideal of two domains $A$ and $B$, then $I$ is stable as an ideal of $A$ if and only if $I$ is stable as an ideal of $B$ since $\Frac A=\Frac B$.

Since 1998,  finitely stable and stable domains have been thoroughly investigated by Bruce Olberding in a series of papers \cite{OG}-\cite{O6}.
In \cite{OG}, he also studied finitely stable rings in the spirit of Rush, extending several results known for stable domains.
Our paper heavily relies on Olberding's work. 
We thank B. Olberding for his valuable help. 
Also, as he communicated to us,  his articles \cite{O6, O1, O2} contain some errors.

Of course, when $R$ is a Noetherian  ring, stability and finite stability coincide, but in general these two classes of rings are distinct, even  if $R$ is an integrally closed domain: in this case $R$ is finitely stable if and only if it is Pr\"ufer, that is, each nonzero  finitely generated ideal of $R$ is invertible. Indeed, a domain $R$ is integrally closed if and only if $R=E(I)$ for each nonzero finitely generated ideal $I$.  However, a valuation domain is stable if and only if it is \emph{strongly discrete}, that is, each nonzero prime ideal is not idempotent \cite[Proposition 7.6]{BS2}. Thus a  valuation domain that is not strongly discrete is finitely stable, but not stable.

 A domain $R$ is finitely stable if and only if it is locally finitely stable \cite[Proposition 7.3.4]{fhp}. Actually, if $I$ is a  stable ideal of $R$, then $I_S$ is a  stable ideal of $R_S$ for each multiplicative part $S\sub R$. 
 
 Recall that a domain $R$ has \emph{finite character} if each nonzero element of $R$ is contained at most in finitely many maximal ideals. A finitely stable domain need not have finite character, since any Pr\"ufer domain is finitely stable. On the other hand,  a  domain is stable if and only if it is locally stable and has finite character \cite[Theorem 3.3]{O2}.

We denote by $R'$ the integral closure of a domain $R$.

Olberding characterized finitely stable domains as follows:

\begin{theorem}\label{OlbCharFS}\cite[Corollary 5.11]{OG} 
A domain $R$ is finitely stable if and only if it satisfies the following conditions:
\begin{enumerate} 
\item 
$R'$ is a quadratic extension of $R$;
\item
$R'$ is a Pr\"ufer domain;
\item Each maximal ideal of $R$ has  at most $2$ maximal ideals of $R'$ lying over it.
\end{enumerate}
\end{theorem}
 	
Recall that a domain $D$ is a {\em quadratic extension} of a domain $R$ if for each $x,y\in D$ we have $xy \in xR+yR+R$.   
Olberding  also proved that, in the local one-dimensional case, stability and finite stability are equivalent provided the maximal ideal is stable:

\begin{proposition}  \label{Olb1dims} \cite[Theorem 4.2]{O8} Let $R$ be a local one-dimensional domain. The following conditions are equivalent:
\begin{enumerate}
\item[(i)] $R$ is stable;
\item[(ii)] $R$ is finitely stable  with stable maximal ideal;
\item[(iii)] $R'$ is a quadratic extension of $R$ and $R'$ is a Dedekind domain with at most two maximal ideals.
\end{enumerate}
\end{proposition}

Stability is related to divisoriality and to the $2$-generator property.
Recall that an ideal $I$ of a domain $R$ is \emph{divisorial} if $I\ne(0)$ and $I=I_v=(R~:~(R:~I))$. A domain $R$ is called \emph{divisorial} if each nonzero ideal of $R$ is divisorial, and it is called \emph{totally divisorial} if each overring of $R$ is divisorial.
 An ideal $I$ of $R$ is called \emph{$2$-generated} if $I$ can be generated by  two elements. The domain $R$  is {\em $2$-generated} if each finitely generated ideal of $R$ is $2$-generated.

A  domain $R$ is stable and divisorial if and only if it is totally divisorial \cite[Theorem 3.12]{O5}. Also,
any stable Noetherian domain is one-dimensional \cite[Proposition 2.1]{sv2}, and a Noetherian domain is stable and divisorial (i.e., totally divisorial) if and only if it is $2$-generated (\cite[Theorem 3.1]{O6} and \cite[Theorem 7.3]{BS2}). The $2$-generator property for Noetherian domains is strictly stronger than stability. The first example of a stable Noetherian domain that is not $2$-generated (equivalently, it is not divisorial) was given in \cite[Example 5.4]{sv2}. Several other examples can be found in \cite[Section 3]{O5}.

A \emph{Mori domain} is a domain with the ascending chain condition on divisorial ideals. This is equivalent to the property that each nonzero ideal $I$ of $R$ contains a  finitely generated nonzero ideal $J$ such that $(R:I)=(R:J)$, that is, $I_v=J_v$ \cite[Theorem 2.1]{Bar}.
Clearly Noetherian domains are Mori. For the main properties of Mori domains, see the survey  \cite{Bar} and the references there.
A nonzero ideal $I$ of an integral domain $R$ is {\em $2$-$v$-generated} if $I$ contains a $2$-generated  ideal $J$ such that
$(R:I)=(R:J)$, and  $R$ is  {\em $2$-$v$-generated} if  each nonzero ideal of $R$ is $2$-$v$-generated. Of course, a  $2$-$v$-generated domain is Mori.  However, if each divisorial ideal of $R$ is principal (hence $2$-$v$-generated), then $R$ is not necessarily Mori (see \cite[page 561]{MZ}).
Clearly, a Mori $2$-generated domain is $2$-$v$-generated.

A Mori domain $R$ satisfies the ascending chain condition on principal ideals (for short, accp), and so it is  \emph{Archi\-me\-dean}, that is,    $\bigcap_{n\geq 0}r^nR=(0)$, for each nonunit $r\in R$. Indeed, a domain $R$ satisfies accp if and only if $\bigcap_{n\ge1}(\prod_{i=1}^n r_iR) =(0)$ for any nonunits $r_i\in R$, equivalently $\bigcap_{n\ge1}a_nR=(0)$ if the sequence of principal ideals $a_nR$  is strictly decreasing.
Besides accp domains, the class of  Archimedean domains includes also one-dimensional domains \cite[Corollary 1.4]{Ohm} and completely integrally closed domains
\cite[Corollary 13.4]{gilmer}. We recall that a domain $R$ is completely integrally closed if and only if $R=E(I)$ for each nonzero ideal $I$. Hence completely integrally closed domains are integrally closed and the converse holds in the Noetherian case. A completely integrally closed stable domain is Dedekind.

Here are our main results:
{\em
\begin{enumerate}
\item If $R$ is an Archimedean  stable domain such that $R^\prime$ is local, then $R$ is one-dimensional (Corollary \ref{stable1dim}).
\item A  domain $R$ is stable and one-dimensional if and only if it is finitely stable and  Mori (Theorem \ref{MoriFS}).
If $R$ satisfies these two equivalent conditions, then each overring of $R$ also satisfies these conditions and it is $2$-$v$-generated.
\end{enumerate}
}

If $R^\prime$ is not local, an Archimedean local stable domain $R$ need not be one-dimensional. Indeed, we present  examples of Archimedean local stable domains of dimension $n$, for each $n\geq 1$, in Part 2, section \ref{examples}. 

A class of one-dimensional local domains that are stable and not Noetherian was  constructed by Olberding in  \cite[Theorems 4.1 and 4.4]{O7} (see also \cite[Theorem 3.10]{O5}). By our results, all these domains are new examples of one-dimensional Mori domains.

We thank T. Dumitrescu for pointing out some errors in  previous versions of this paper.


\section{The one-dimensional case}
In the following, $R$ is an integral domain that is not a field. By an {\it ideal} we mean an integral ideal. 

The following construction, due to Olberding, is basic for our paper.

\begin{construction}\label{olbcons1}\cite[Section 4]{O2}
	\rm Let $(R,M)$ be a local  domain.
	Set $R_i=\{0\}$ for $i<0$, $R_0=R$ and $M_0=M$. Define inductively for $n > 0$:
	$R_n= R_{n-1}$ if $R_{n-1}$ is not local, and $R_n=E(M_{n-1})=(M_{n-1}:M_{n-1})$ if
	$R_{n-1}$ is local with maximal ideal denoted by $M_{n-1}$.
	Set $T=\bigcup_{n\geq 0}R_n$.
	
	Thus we have:
	\begin{enumerate}
		\item[(a)]
		If there exists an integer $k>0$ such that
		$R_k$ is not local, but $R_i$ is local for $0\le i<k$, then $R_n=R_k$ for all $n\ge k$, and $T=R_k$.
		\item[(b)]
		If $R_n\subsetneq R_{n+1}$ for all $n\ge0$, all the rings $R_n$ are local.
	\end{enumerate}
\end{construction}
	
We will use  repeatedly the following  theorem of Olberding.

\begin{theorem} \cite[Corollary 4.3, Theorem 4.8]{O2} {\rm and its proof, and} \cite[Theorem 5.4]{OG} \label{Olb} Let $R$ be a finitely stable local domain with stable maximal ideal $M$.  With the notation of Construction \ref{olbcons1} we have:

\begin{enumerate}
		\item[(1)] Each $R_n$ is finitely stable with stable maximal ideals, and there exists an element $m\in M$ such that $M=mR_1$. Moreover, for $k\geq 1$, if $R_k$ is local with maximal ideal $M_k$, then  $M_k=mR_{k+1}=MR_{k+1}$, and if $T$ is local, then its maximal ideal is $mT=MT$.   

\item[(2)] Each $R_n$ is a finitely generated $R$-module, thus $T$ is an integral extension of $R$. 
\end{enumerate}

We also have:
	
\begin{enumerate}
\item[(a)]  If  $T=R_n$ for some $n\geq 0$, then $T$ is a finitely generated $R$-module, and $T$ has at most two maximal ideals.
		
\item[(b)] If $T\neq R_n$ for all $n\geq 0$,  then  $T$ is local.

\item[(c)] The maximal ideals of $T$ are principal, and the Jacobson radical of $T$ is equal to
$mT=MT$, where $mR_1=M$.
\end{enumerate}
In addition,
 if $R$ is  a  stable domain, then $T$ is equal to the integral closure $R^\prime$ of $R$, and $R'$ is a strongly discrete Pr\"ufer domain.
 	\end{theorem}

 In the one-dimensional case we have:

\begin{corollary} \label{DVR} Let $R$ be a one-dimensional finitely stable local domain with stable maximal ideal $M$. Then $R$ is stable, and in the setting of Theorem \ref{Olb}, $T=R^\prime$ is a principal ideal domain with at most two maximal ideals. 
Hence, if $T$ is local, in particular, if $T\neq R_n$ for each $n\geq 0$, then $T$ is a DVR.
\end{corollary}
\begin{proof} $R$ is stable by Proposition \ref{Olb1dims}, so $T=R'$.  Since $R'$ is one-dimensional with principal maximal ideals, $R'$ is a principal ideal domain by \cite[Corollary 37.9]{gilmer}.
\end{proof}

\begin{proposition} \label{charab} In the setting of Theorem \ref{Olb},
 	$T\neq R_n$ for each $n\geq 0$ if and only if $T$ is a finite $R$-extension (that is, $T$ is a finitely generated $R$-module).
 	Hence $T=R_n$ for some $n\geq 0$ if and only if $T$ is not  a finite $R$-extension.
   	(Recall that if $R$ is stable, then $T=R'$.)
\end{proposition}

\begin{proof}
If $T=R_n$ for some $n\geq 0$, then $T$ is a finitely generated  $R$-module by Theorem \ref{Olb} (a).
Conversely, assume that $T$ is generated as an $R$-module by a finite subset $F$ of $T$. Then there exists an integer $n\ge0$ such that $F\sub R_n$, implying that $T=R_n$. 
\end{proof}

Denote by $\U(A)$ the set of units of a domain $A$.

\begin{remark}
In the setting of Theorem \ref{Olb}, for any integer $n\geq 0$  we have $\U(T)\cap R_n=\U(R_n)$, 
since $T$ is an integral extension of $R_n$.
\end{remark}

\begin{lemma}\label{tmn}
	Let $R$ be a finitely stable local domain with stable maximal ideal. In the setting of Theorem \ref{Olb}, if $T$ is local, in particular, if condition (b) holds, we have: 
\begin{enumerate} 
\item 
For each $n\geq 0$, 
	$(R:_T m^n)=(R:_T M^n)=R_n$;   equivalently, $Tm^n\cap R=R_nm^n$   (here $M^0=R$).  
\item 
Let $r=um^n$ be a nonzero element of $R$, where $u\in\U(T)$, and $n\ge0$. Then $(R:_T r)=R_n$.
\end{enumerate}
\end{lemma}

\begin{proof}
(1) 
We prove the equality $(R:_T m^n)=R_n$ by induction on $n$ starting with $n=0$. Let $n>0$.
Since $M=R_1m$, by applying the induction assumption to  $R_1$  replacing $R$ we obtain  that:
$$(R:_T m^n)=(M:_T m^n)=(R_1m:_T m^n)=(R_1:_T m^{n-1})=R_n.$$
Also $M^n=(R_1m)^n=R_1m^n$. Since $R_n=(R:_T m^n)$ and $R_1 \subseteq R_n$, we obtain 
$$
R_n\subseteq  (R:_T R_1m^n)=(R:_T M^n)\subseteq (R:_T m^n)=R_n, 
$$
so  $(R:_T M^n)=R_n$.
 
(2)
By item (1) we have $u\in R_n$,  and also:
$$
(R:_T r)=(R:_T um^n)=((R:_T m^n):u)=(R_n:_K u)=R_n,
$$
where $K=\Frac R$, since $u\in \U(R_n)$.
\end{proof}

\begin{lemma}\label{Tint} 
	Let $(R,M)$ be a finitely stable local domain with stable maximal ideal. In the notation of \ref{olbcons} assume that $T$ is local. Then
	$$
	\Bigl(\bigcap_{n\ge0} m^nT\Bigr)^2\sub\bigcap_{n\ge0} m^nR.
	$$
\end{lemma}	  

\begin{proof}
	By Lemma \ref{tmn} (1), we have for all $n\ge0$:
	$$\Bigl(R\cap\bigcap_{k\ge0} m^kT\Bigr)^2\subseteq  (R\cap m^nT)^2=(m^n R_n)^2=m^n(m^n R_n )\subseteq  m^nR,$$ 
	so 	$\Bigl(R\cap\bigcap_{k\ge0} m^kT\Bigr)^2\subseteq \bigcap_{n\ge0} m^nR$. 
	
Now let $s,t\in\bigcap_{n\ge0} m^nT$.  Again by Lemma  \ref{tmn} (1), we have $sm^e,tm^e\in R$ for a sufficiently large integer $e$. Thus 
$$(sm^e)(tm^e) \in  \Bigl(R\cap\bigcap_{n\ge0} m^nT\Bigr)^2\subseteq \bigcap_{n\ge0} m^nR.$$
It follows that $st=\frac{(sm^e)(tm^e)}{m^{2e}}\in \bigcap_{n\ge0} m^nR$. Hence
$\Bigl(\bigcap_{n\ge0} m^nT\Bigr)^2\sub\bigcap_{n\ge0} m^nR$.
\end{proof}

\begin{theorem}\label{Tlocal} 
	Let $R$ be an Archimedean finitely stable local domain with stable maximal ideal. In the notation of \ref{olbcons}, if $T$ is local, in particular, if condition (b) of Theorem \ref{Olb} holds, then  $R$ is one-dimensional.	 
\end{theorem}

\begin{proof} 
	By Theorem \ref{Olb}, the maximal ideal of $T$ is $mT$, $m\in M$. Let $Q=\bigcap_{n\ge0} m^nT$. By Lemma \ref{Tint}, $Q^2\sub \bigcap_{n\ge0} m^nR=(0).$ Hence  $Q=(0)$.
	By \cite[Theorem (7.6) (a) and (c)]{gilmer}, $Q$ is the largest non-maximal prime contained in $mT$. Thus $T$ is one-dimensional, and so is $R$, as $T$ is an integral extension of $R$.
\end{proof}

\begin{corollary}\label{stable1dim}
	Let $R$ be an Archimedean stable  domain satisfying one of the following two conditions:
	\begin{enumerate} 
\item[\rm{(a)}] 
$R$ is local and $R^\prime$ is not a finitely generated $R$-module; 
 
\item[\rm{(b)}]
$R'$ is local.
\end{enumerate}
Then $R$ is one-dimensional.
\end{corollary}

\begin{proof} 
(a)  Since $R$ is local and stable, we have $T=R'$.  By Proposition \ref{charab}, condition (a) here is equivalent to condition (b) of Theorem \ref{Olb}. Thus $T=R^\prime$ is local and so $R$ is one-dimensional by Theorem  \ref{Tlocal}. 

(b) Since $R^\prime$ is local, also $R$ is local. As in (a), since $R$ is also stable, we have $T=R^\prime$, and again $R$ is one-dimensional by Theorem  \ref{Tlocal}. 	
\end{proof}

As already mentioned, if $R^\prime$ is not local, an Archimedean local stable domain $R$ need not be one-dimensional. Indeed, we present  examples of Archimedean local stable domains of dimension $n$, for each $n\geq 1$, in Part 2, section \ref{examples} (that is, \cite[Section 5]{GR}).

\section{The $2$-$v$-generator property}

In the Noetherian case, the next theorem was proved by Sally and Vasconcelos \cite[Theorem 2.4]{sv1}. Olberding proved that the hypotheses of the theorem imply Noetherianity.

\begin{theorem} \label{a2gen} \cite[Proposition 4.5]{O1}
	Let $R$ be a one-dimensional stable domain.
	If $R'$ is a finite $R$-extension, then each ideal of $R$ is $2$-generated.
\end{theorem}

In Proposition \ref{1dim2v} below,  we state that a  stable one-dimensional domain $R$ is \emph{$2$-$v$-generated}, that is, for each nonzero ideal $I$ there are two elements $x, y\in I$ such that $I_v= \langle x, y \rangle_v$; thus $R$ is Mori. We start with the following notation:

\begin{notation}\label{v}
	In the setting of Theorem \ref{Olb},  assume that the domain $R$ is one-dimensional and that $T$ is local (in particular, $T$ is local if condition (b) holds). As $T$ is a DVR (Corollary \ref{DVR}) with maximal ideal $mT$, we denote by $\vv$  the discrete valuation of $T$ such that $\vv(m)=1$.
\end{notation}

\begin{lemma}\label{2v}
In the setting of Theorem \ref{Olb},  assume that the domain $R$ is one-dimensional and that $T$ is local. Then, by using Notation \ref{v}, we have:
	
	\begin{enumerate}
		\item
		Let $r$ be a nonzero element of $R$. Then:
		$$(R:_T r)=R_{\vv(r)}.$$
		
		\item
		Let $I$ be a nonzero ideal of $R$, and let $a$ be an element of minimal value $\vv(a)=k$ in $I$.
		Then:
		$$(R:_T I)=R_k.$$
	\end{enumerate}
\end{lemma}

\begin{proof}
(1)  This follows from Lemma \ref{tmn} (2).
 
(2)
By item (1), we have
$$(R:_T I)=\bigcap_{r\in I\setminus \{0\}} (R:_T r)=\bigcap_{r\in I\setminus \{0\}} R_{\vv(r)}=R_k.$$
\end{proof}

From Lemma \ref{tmn} (1) we obtain:

\begin{lemma}\label{vRk}
In the setting of Theorem \ref{Olb},  assume that the domain $R$ is one-dimensional and that $T$ is local. Then, in the notation \ref{v}, we have
 for all $k\ge0$:  
$$\{r\in R\, | \, \vv(r)\ge k\}= R\cap m^kT= R_k m^k.$$	 
\end{lemma}	  

\begin{proposition} \label{1dimloc2v} A one-dimensional  local stable domain $R$ is  $2$-$v$-generated; hence $R$ is a Mori domain.
\end{proposition}

\begin{proof} 	In case (a) of Theorem \ref{Olb},  every ideal of $R$ is $2$-generated by Theorem \ref{a2gen}.
	
	Assume condition  (b) of Theorem \ref{Olb}, and use Notation \ref{v}. Let $I\ne R$ be a nonzero ideal of $R$.
	Since $T$ is a DVR, there exists a nonzero element $t\in T$ of maximal value $\vv(t)$ such that $\frac 1 t I\sub R$. Let $J=\frac 1 t I$, so  $(R:J)\subseteq T$. Since $\frac 1 m\notin T$,  there exists a nonzero element $a_1\in J$ such that $\frac {a_1}m\notin R$. Let $a_2$ be an element of minimal value $k$ in $J$. If $\frac {a_2}m\notin R$, set $a=a_2$. Assume that $\frac{a_2}m\in R$. If $\vv(a_1)=\vv(a_2)$, set $a=a_1$. Otherwise $\vv(a_1)>\vv(a_2)$, so $\vv(a_1+a_2)=\vv(a_2)$ and $\frac {a_1+a_2} m\notin R$. In this case we set $a=a_1+a_2$.
	In each case, $a$ is a nonzero element of minimal value $k$ in $J$ such that $\frac a m\notin R$. Thus $a=um^k$, where $u\in\U(R_k)\sm R_{k-1}$, by Lemma \ref{tmn} (1).
	
	Since $(R:J)\sub T$ and $\frac {1}{um}\notin T$, there exists an element $b\in J$ such that $\frac b{um}\notin R$.  We show that \mbox{$(R:\{ a,b\})\sub T$}.
	
	 If $x$ is an element in $(R:\{a,b\})\sm T$, we have $x=\frac 1{vm^i}$, where $v\in\U(T)$ and $i>0$. Thus $\frac 1{vm}a, \frac 1{vm}b\in R$.
	Since $\frac 1{vm} a=\frac uv m^{k-1}\in R$, we have $\frac uv\in \U(R_{k-1})$ by Lemma \ref{tmn} (1). Since $\vv(b)\ge k$, we have   $\vv(\frac{b}{vm})\ge k-1$. As $\frac{b}{vm}\in R$, we obtain  by Lemma \ref{vRk} that
	$\frac{b}{vm}\in R_{k-1}m^{k-1}$. Hence, $\frac b{um}=\frac vu\frac b{vm} \in R_{k-1}m^{k-1}\subseteq R$, a contradiction. It follows that $(R:\{ a,b\})\sub T$.
	
	Since $a$ is of minimal value in $J$, by Lemma \ref{2v} (1)-(2), we have \mbox{$(R:_T J)=(R:_T a)$}.
	
	Hence $(R:J)\sub (R:\{ a,b\})=(R:_T\{ a,b\})=(R:_T J) \subseteq (R:J)$,	
	so $(R:J)=(R:\{ a,b\})$. Thus  $J$ is $2$-$v$-generated and so is $I=tJ$. We conclude that $R$ is $2$-$v$-generated.
\end{proof}

\begin{corollary} Let $R$ be an Archimedean stable domain such that $R^\prime$ is local (in particular, assume that condition (b) of Theorem \ref{Olb} holds). Then $R$ is a one-dimensional  Mori domain.
\end{corollary}
\begin{proof} $R$ is local and one-dimensional by Corollary \ref{stable1dim}. Hence $R$ is Mori by Proposition \ref{1dimloc2v}.
\end{proof}

In Proposition \ref{1dim2v} below we globalize Proposition \ref{1dimloc2v}.

\begin{lemma}\label{2vs}
	Let $S$ be a multiplicative subset of an integral domain $R$. If $I$ is a $2$-$v$-generated nonzero ideal of $R$, then the ideal $IR_S$ of $R_S$ is $2$-$v$-generated. Hence, if $R$ is $2$-$v$-generated,  also $R_S$ is $2$-$v$-generated.
\end{lemma}

\begin{proof}
	There exists a $2$-generated subideal $J$ of $I$ such that \mbox{$(R:J)=(R:I)$}. Since $(R:J)R_S=(R_S:JR_S)$, we have 
	\mbox{$(R_S:I R_S)=(R_S:JR_S)$} and so the ideal $IR_S$ of $R_S$ is $2$-$v$-generated.
\end{proof}

\begin{lemma}\label{assoc}
	Let $(R,M)$ be a  local one-dimensional domain,  and let $a,b\in M$ be two nonzero elements.
		Then  each element in $a+Rb^k$ is associated with $a$ for all sufficiently large integers $k$.
	\end{lemma}

\begin{proof}
		Since $R$ is local and one-dimensional, we have $M=\sqrt {aM}$, so $b^k\in aM$ for each sufficiently large integer $k$. Hence
		for all $r\in R$ we have $a+rb^k=a(1+ r(\frac {b^k}a))$, where $1+r\frac {b^k}a$ is a unit in $R$ since $\frac {b^k}a \in M$.
\end{proof}

\begin{proposition}\label{2vg}
	Let $R$ be a one-dimensional domain of finite character. The following conditions are equivalent:
	\begin{enumerate}
 \item[(i)]  $R$ is $2$-$v$-generated;
 \item[(ii)] $R$ is
	locally $2$-$v$-generated. 
	\end{enumerate}\end{proposition}

\begin{proof}
(i) $\ra$ (ii) If $R$ is $2$-$v$-generated,  then $R$ is
	locally $2$-$v$-generated by  Lemma \ref{2vs}.

(ii) $\ra$ (i) Assume that $R$ is locally $2$-$v$-generated. We prove that each  nonzero ideal $I\ne R$ of $R$ is $2$-$v$-generated.
Since $R$ has finite character,
there are just finitely many maximal ideals containing $I$, say
$M_1,\dots, M_e$, which we assume to be distinct.  For each $1\le i\le e$,  the domain $R_{M_i}$ is  $2$-$v$-generated, so
there exist  nonzero elements $a_i,b_i$ in $I$ such that $(R_{M_i}:~I)=(R_{M_i}:\{ a_i,b_i\})$.
There exist pairwise comaximal elements $m_i\in M_i$, for $1\le i\le e$. By the Chinese Remainder Theorem, for each positive integer $k$ there exists an element $a \in I$ such that  we have in $R$:
$$
a\equiv a_i\mod I m_i^k
$$
for $1\le i\le e$.
By Lemma \ref{assoc}, we may choose  $k$  sufficiently large such that for each $i$ the elements $a$ and $a_i$ are associated in $R_{M_i}$, so $(R_{M_i}:I)=(R_{M_i}:~\{ a_i,b_i\})=(R_{M_i}:\{ a,b_i\})$.

Let $N_q\ (q=1,2,\dots, f)$ be the maximal ideals containing $a$ but not $I$. There exist pairwise comaximal elements $n_i\in M_i$, for $1\le i\le e$ that belong to no maximal ideal $N_q$. Also  there exists an element $c\in I$ that belongs to no  ideal $N_q$. By the Chinese Remainder theorem, for each positive integer $j$ there exists an element $b\in I$ such that $b\equiv b_i \mod In_i^j$ for each $1\le i\le e$,  and $b\equiv c \mod IN_q$ for each ideal $N_q$.  Hence $b\notin N_q$ for all $1\le q\le f$.
By Lemma \ref{assoc}, for a sufficiently large  integer $j$ and for each $i$, the elements $b$ and $b_i$ are associated in $R_{M_i}$, so $(R_{M_i}:I)=(R_{M_i}:\{ a,b\})$ for all $1\le i\le e$.

Let $M$ be a maximal ideal of $R$. If $M$ contains $I$, thus $M=M_i$ for some integer $1\le i\le e$, then
$(R_{M_i}:I)=(R_{M_i}:\{ a, b\}).$
If $M$ contains $a$ but not $I$, then $M=N_q$ for some integer $1\le q\le f$, so $b\notin M$. Thus
$(R_M:~I)=R_M=(R_M:\{ a, b\})$. If $M$ does not contain $a$, then again $(R_M:I)=R_M=(R_M:\{ a, b\})$.
Thus for each maximal ideal $M$ of $R$ we have $(R_M:I)=(R_M:\{ a, b\})$. Hence
$$
(R:I)=\bigcap_M (R_M:I)=\bigcap_M (R_M:\{a,b\})=(R:\{a,b\}),
$$
where $M$ runs over all the maximal ideals of $R$. We conclude that  $I$ is $2$-$v$-generated, so the domain $R$ is $2$-$v$-generated.	 
\end{proof}	  

\begin{corollary}\label{stab2v}
	A  one-dimensional stable domain is $2$-$v$-generated if and only if it is locally  $2$-$v$-generated	 
\end{corollary}

\begin{proof} 
	Indeed, a stable  domain has finite character.	 
\end{proof}	

A locally $2$-$v$-generated domain $R$ is not necessarily $2$-$v$-generated even if $R$ is one-dimensional. For example, if $R$ is an almost Dedekind domain that is not Dedekind, then
$R$ is locally a DVR, but $R$ is not Mori since an almost Dedekind and Mori domain  is Dedekind. For a positive result, see  Proposition \ref{onedim2v} below.

\begin{lemma} \label{1morifc}
A one-dimensional Mori domain  has finite character.
\end{lemma} 

\begin{proof} If $R$ is Mori and one-dimensional, every maximal ideal of $R$ is divisorial \cite[Theorem 3.1]{Bar}. By \cite[Theorem 3.3 (c)]{Bar}, a Mori domain is an intersection of finite character of the localizations at its maximal divisorial ideals. It follows that $R$ has finite character.
\end{proof}

\begin{proposition}\label{onedim2v}
	Let  $R$ be a one-dimensional domain. The following conditions are equivalent:
	\begin{enumerate}
		\item[(i)] $R$ is $2$-$v$-generated;
		\item[(ii)] $R$ is locally $2$-$v$-generated and $R$ has finite character.
	\end{enumerate}
\end{proposition}

\begin{proof}
	(i) $\ra$ (ii) $R$ is locally $2$-$v$-generated by Lemma \ref{2vs} and has finite character by Lemma \ref{1morifc}.
	
	(ii) $\ra$ (i)
	See Proposition \ref{2vg}.
	\end{proof}

\begin{proposition} \label{1dim2v} A one-dimensional   stable domain $R$ is  $2$-$v$-generated; hence $R$ is Mori.
\end{proposition}

\begin{proof}
	Since $R$ is locally stable,  $R$ is locally $2$-$v$-generated by Proposition \ref{1dimloc2v}. Thus $R$ is $2$-$v$-generated by Corollary \ref{stab2v}.
\end{proof}

The stability assumption  in Propositions \ref{1dimloc2v} and  \ref{1dim2v} cannot be relaxed to finite stability. Indeed, let $R$ be a one-dimensional valuation domain that is not a DVR. Thus $R$ is finitely stable, but $R$  is neither Mori, nor stable (the maximal ideal of $R$ is not stable); see \cite[Example 3.3]{O1}. On the other hand,  we prove below that a one-dimensional finitely stable Mori domain is stable (Proposition \ref{fstablemori}).

\section{The Mori case}
In this section, we give a characterization of one-dimensional stable domains. We need a few preliminary results.

\begin{proposition}\label{radical} 
Let $I$ be a stable ideal of an integral domain $R$. Then $I_v=I(I_v:I_v)$ is stable, and $(I_v)^2\sub I$.
\end{proposition}

\begin{proof} 
Let $D=(I_v:I_v)$. Thus $(I:I)\subseteq (I_v:I)=(I_v:I_v)= D$. Since $I$ is an invertible ideal of $(I:I)$ and $(I:I) \subseteq  D$, it follows that $ID$ is an invertible ideal of $D$. As $D$ is a fractional divisorial ideal of $R$, we obtain that $ID$ is a fractional divisorial ideal of $R$. Hence $I_v\subseteq ID$, so $I_v=ID$ since $I_v$ is an ideal of $D$.
 Thus $I_v=ID$ is invertible in $D=(I_v:I_v)$, that is, $I_v$ is a stable ideal of $R$. Also $(I_v)^2=I_v(ID)=(I_vD)I=I_vI \subseteq I$.
\end{proof}	  

\begin{corollary}\label{vfiniteStable} \cite[Lemma 2.7]{GP}
	In a finitely stable domain all the $v$-finite divisorial ideals are stable. In particular, all the divisorial ideals of a finitely stable Mori domain are stable.
\end{corollary}

A nonzero ideal $I$ of a domain is called a \emph{$t$-ideal} if  
$I=\bigcup J_v$, where $J$ runs over all  finitely generated subideals of $I$. Divisorial ideals are $t$-ideals, and in a Mori domain each $t$-ideal is divisorial.

\begin{corollary} \label{radicalv}
\	
\begin{enumerate} 
\item 

A stable radical  ideal is divisorial.

{\rm(Cf. \cite[Corollary 4.13]{O2}. Here we do not assume that the domain $R$ is stable).} 
 
\item
If $I$ is a radical ideal and each finitely generated subideal of  $I$ is stable, then $I$ is a $t$-ideal.

\item
Each nonzero radical ideal of a finitely stable domain is a $t$-ideal.

\item
All the nonzero radical ideals of a finitely stable Mori domain are divisorial and stable.
\end{enumerate} 
\end{corollary}	  

\begin{proof} 
\begin{enumerate} 
\item 
Let $I$ be a  stable radical ideal of $R$. By Proposition \ref{radical}, we have  $(I_v)^2\sub I$, so  $I_v\sub I$ as the ideal $I$ is radical. Hence $I=I_v$ is a divisorial ideal. 
 
\item 
If $J$ is a nonzero finitely generated subideal of $I$,
then $(J_v)^2\sub J\sub I$ by Proposition \ref{radical}. Since the ideal $I$ is radical, we obtain $J_v\sub I$, so $I$ is a $t$-ideal.

\item follows from (2).

\item
All the radical ideals of a Mori domain are divisorial by item (2), so they are also stable by Corollary \ref{vfiniteStable}.

\end{enumerate}	 
\end{proof}	  

\begin{proposition}\label{fstablemori} 
A one-dimensional finitely stable Mori domain is stable.	 
\end{proposition}	  

\begin{proof} 
For each maximal ideal $M$ of $R$,  $R_M$ is Mori and finitely stable. Hence $MR_M$  is divisorial (Corollary \ref{radicalv} (3)) and so stable
(Corollary \ref{vfiniteStable}).  By Proposition \ref{Olb1dims}, $R_M$ is stable. Since  $R$ has finite character (Lemma \ref{1morifc}), $R$ is stable by \cite[Theorem 3.3]{O2}.  
\end{proof}	  

Actually, as  shown in Theorem \ref{MoriFS} below,  a finitely stable Mori domain is one-dimensional, so it is stable and $2$-$v$-generated (Propositions \ref{fstablemori} and \ref{1dim2v}). 

The following lemma is known;  we give a proof for lack of a reference.

\begin{lemma}\label{IIMori} 
	Let $I$ be  a divisorial ideal of a Mori domain $R$. Then the domain $(I:I)$ is Mori.	 
\end{lemma}	

\begin{proof} 
	Let $J_1 \subseteq  J_2 \subseteq  \dots$ be an infinite increasing sequence of divisorial ideals of the domain $(I:I)$. Since $I$ is  a divisorial ideal of $R$, the domain $(I:I)$ is  a fractional divisorial ideal of $R$, so $J_1,J_2,\dots$ are fractional divisorial ideals of $R$. Let $c$ be a nonzero element of $I$. Then
	$cJ_1 \subseteq  cJ_2\subseteq  \dots$ is a an increasing sequence of divisorial ideals of $R$, so $cJ_n=cJ_{n+1}$ for $n\gg0$. Thus the sequence $J_1 \subseteq J_2\subseteq  \dots$  stabilizes, implying that $(I:I)$	is Mori. 
\end{proof}

\begin{proposition}\label{fslMori} 
	Let $(R,M)$ be a finitely stable local Mori domain. If $T$ is a finite extension of $R$, then $R$ is  one-dimensional, stable and every ideal of $R$ is $2$-generated, thus the domain $R$ is Noetherian. (see  \ref{olbcons} for the definition  of $T$).	  
\end{proposition}	  

\begin{proof} 
	By Corollary \ref{radicalv} (4), the maximal ideal $M$ of $R$ is divisorial and stable.
	We use the setting of Theorem \ref{Olb}. By Proposition \ref{charab}, $T=R_n$ for some integer $n\ge0$. By Lemma \ref{IIMori}, the domain $R_1=(M:M)$ is Mori. By induction, $R_k$ is Mori for all $k\ge0$, so $T=R_n$ is a Mori domain. Since $T$ has principal maximal ideals (Theorem \ref{Olb} (c)), $T$ is one-dimensional \cite[Theorem 3.4]{Bar}. So $R$ is one-dimensional. By Proposition \ref{fstablemori},  $R$ is stable. By Theorem \ref{a2gen}, every ideal of $R$ is $2$-generated.	 
\end{proof}	  

\begin{proposition}\label{MoriLoc} 
	Let $(R,M)$ be a local domain.
	The following conditions are equivalent:
	\begin{enumerate} 
		\item[\rm(i)] 
		$R$ is one-dimensional and  stable.  
		
		\item[\rm(ii)] 
		$R$ is finitely stable and Mori. 
	\end{enumerate}	 
\end{proposition}	  

\begin{proof} 
	(i) $\ra$ (ii) See Proposition \ref{1dimloc2v}.
	
	(ii) $\ra$ (i)
	By Corollary \ref{radicalv} (4), the maximal ideal $M$ of $R$ is stable.
	By Proposition \ref{fslMori}, we have to consider just the case (b) of Theorem \ref{Olb}.  In this case, by Theorem \ref{Tlocal},  $R$ is one-dimensional. By Proposition \ref{fstablemori}, $R$ is stable.
\end{proof}	  

In the next theorem we globalize Proposition \ref{MoriLoc}:

\begin{theorem}\label{MoriFS} 
	Let $R$ be an integral domain.
	The following two conditions are equivalent:
	\begin{enumerate} 
		\item[\rm(i)] 
		$R$ is one-dimensional and stable.
		
		\item[\rm(ii)] 
		$R$ is finitely stable and Mori. 
	\end{enumerate}	 
	Moreover, if $R$ satisfies these two equivalent conditions, then  every overring of  $R$ also satisfies the two conditions,  every overring of $R$ is $2$-$v$-generated, and $R'$ is a Dedekind domain.	 
\end{theorem}	  

\begin{proof} 
	(i) $\ra$ (ii)	Since $R$ is locally stable, we obtain that $R$ is locally Mori by Proposition \ref{MoriLoc}.  Since $R$ has finite character, it follows that $R$ is Mori \cite[Theorem 2.4]{Bar}.  
	
	(ii) $\ra$ (i)
	Since $R$ is locally finitely stable and locally Mori, it follows from Proposition  \ref{MoriLoc} that  $R$ is one-dimensional and locally stable. Since $R$ has finite character (Lemma \ref{1morifc}), $R$ is stable.
	
	Assume that $R$ satisfies the two conditions.  Let $D$ be an overring of $R$. Since $R$ is one-dimensional and $R'$ is Pr\"ufer (as $R$ is stable), it follows that each overring of $R$ is  is one-dimensional by \cite[Theorem 6]{gpp}. Since $R$ is stable, each overring of $R$ is stable. A one-dimensional stable domain is $2$-$v$-generated by Proposition \ref{1dim2v}. Finally, $R'$ is Pr\"ufer and Mori, so it is Dedekind (alternatively, this follows from that a stable one-dimensional Pr\"ufer domain is Dedekind).
\end{proof}	  

In connection with Theorem \ref{MoriFS}, recall that an integral domain is Noetherian $2$-generated if and only if it is  one-dimensional, stable and divisorial (\cite[Theorem 3.1]{O6} and \cite[Theorem 7.3]{BS2}).

However, if we assume just that $R$ is a $2$-$v$-generated domain, then $R$ is not necessarily one-dimensional, and so also not finitely stable. Indeed, any Krull domain is $2$-$v$-generated  \cite[Proposition 1.2]{MZ}.  In addition, it is not true that in a $2$-$v$-generated domain each divisorial ideal is stable.
In fact, if $R$ is a Krull domain, stability coincides with invertibility. Thus each divisorial ideal of a Krull domain $R$ is stable (i.e., invertible) if and only if $R$ is locally factorial \cite[Lemma 1.1]{Bou}. On the other hand, a one-dimensional Krull domain is Dedekind and so each nonzero ideal is divisorial and stable.

In view of this example and of the $2$-generated case, we ask:

\begin{question}
	Let $R$ be a $2$-$v$-generated domain $R$. Are  the divisorial ideals of $R$ $v$-stable? If $R$ is one-dimensional, are the divisorial ideals of $R$ stable?
\end{question}

Recall that an ideal $I$ of a domain $R$ is {\em $v$-invertible} if $(I(R:I))_v=R$ and that
a divisorial ideal $I$ of $R$ is {\em $v$-stable} if $I$ is $v$-invertible in the ring $(I:I)$, that is $(I(I:I^2))_v=(I:I)$.

\stepcounter{npart}
\newpage

\part{On Finitely Stable Domains, II}

\footnotesize{\textit{Keywords and phrases.} {accp, Archimedean domain, completely integrally closed, finite character, finitely stable, locally Archimedean, Mori domain, stable ideal.}

\begin{abstract}
Among other results, we prove the following:
\begin{enumerate}

\item
A locally Archimedean stable domain satisfies accp.

\item
A stable domain $R$  is  Archimedean if and only if every nonunit  of $R$ belongs to a height-one prime ideal of the integral closure $R'$ of $R$ in its quotient field (this result is related to Ohm's Theorem for Pr\"ufer domains).

\item
An  Archimedean  stable domain $R$ is one-dimensional if and only if $R'$ is equidimensional
(generally, an Archimedean stable local domain is not necessarily one-dimensional).  

\item
An Archimedean finitely stable semilocal domain with stable maximal ideals is locally Archimedean, but generally, neither Archimedean stable domains, nor Archimedean semilocal domains are necessarily  locally Archimedean.
\end{enumerate}
\end{abstract}
\maketitle

\section{Introduction}

In the following, $R$ is an integral domain with quotient field $K$ and $R\neq K$.  An overring of $R$ is a domain $T$ such that $R\sub T\sub K$. We denote by $R^\prime$ the integral closure of $R$.
By an {\it ideal} we mean an integral ideal. 

This part deals with Archimedean finitely stable domains. 

We recall that a nonzero ideal $I$ of $R$ is called \emph{stable} if $I$ is invertible in its endomorphism ring $E(I):=(I:I)$. $R$ is \emph{finitely stable} if each nonzero finitely generated ideal is stable and is \emph{stable} if each ideal is stable.

Since 1998,  finitely stable and stable domains have been thoroughly investigated by Bruce Olberding in a series of papers \cite{O3}-\cite{O7}.
Our paper heavily relies on Olberding's work. 
We thank B. Olberding for his valuable help. Also, as he communicated to us,  his articles \cite{O6, O1, O2} contain some errors.

Of course, when $R$ is a Noetherian domain, stability and finite stability coincide, but in general these two classes of rings are distinct, even  if $R$ is integrally closed: in this case $R$ is finitely stable if and only if it is Pr\"ufer, that is, each nonzero finitely generated ideal is invertible. Indeed, a domain $R$ is integrally closed if and only if $R=E(I)$ for each nonzero finitely generated ideal $I$.  However, a valuation domain is stable if and only if it is \emph{strongly discrete}, that is, each nonzero prime ideal is not idempotent \cite[Proposition 7.6]{BS2}. Thus a  valuation domain that is not strongly discrete is finitely stable, but not stable.

 A domain $R$ is finitely stable if and only if $R_M$ is finitely stable, for each maximal ideal $M$ \cite[Proposition 7.3.4]{fhp}. Actually, if $I$ is a  stable ideal of $R$, then $I_S$ is a  stable ideal of $R_S$ for each multiplicative part $S\sub R$.
 
 A domain has \emph{finite character} if each nonzero element is contained at most in finitely many maximal ideals. A finitely stable domain need not have finite character, since any Pr\"ufer domain is finitely stable. On the other hand,  a  domain is stable if and only if it is locally stable and has finite character \cite[Theorem 3.3]{O2}.
 
  We recall that a domain $R$ is called  \emph{Archi\-me\-dean} if $\bigcap_{n\geq 0}r^nR=(0)$, for each nonunit $r\in R$. If $R$ satisfies the ascending chain condition on principal ideals (for short, accp), then $R$ is Archimedean. Indeed,  the domain $R$ satisfies accp if and only if $\bigcap_{n\ge1}(\prod_{i=1}^n r_iR) =(0)$ for any nonunits $r_i\in R$, equivalently $\bigcap_{n\ge1}a_nR=(0)$ if the sequence of principal ideals $a_nR$  is strictly decreasing.
A \emph{Mori domain} is a domain satisfying the ascending chain condition on divisorial ideals, so a Mori domain is Archimedean. Besides accp domains, the class of  Archimedean domains includes also one-dimensional domains \cite[Corollary 1.4]{Ohm} and completely integrally closed domains
\cite[Corollary 13.4]{gilmer}. We recall that a domain $R$ is completely integrally closed if and only if $R=E(I)$ for each nonzero ideal $I$. Hence completely integrally closed domains are integrally closed and the converse holds in the Noetherian case. 
A completely integrally closed stable domain is Dedekind.

In \cite[Theorem 4.8]{GR1} we proved that a domain is stable and one-dimensional if and only if it is Mori and finitely stable. Here, among other results,  we show that an Archimedean stable domain is one-dimensional if and only if $R'$ is equidimensional (Proposition \ref{equidim}). 
The assumption that $R'$ is equidimensional is essential, as shown in Example \ref{ex6}.

As usual, if $\mc P$ is a property of rings, then a ring $R$ is {\em locally} $\mc P$ if $R_M$ is $\mc P$ for each maximal ideal $M$ of $R$. Generally, this does not imply that $R_P$ is $\mc P$ for every prime ideal $P$ even for a local domain (see Example \ref{ex1} for the Archimedean property).  The property $\mc P$ {\em localizes} if  every ring satisfying $\mc P$ is locally $\mc P$. The following properties localize:
stability, finite stability, Mori. However, as it is well-known, the Archimedean property, the accp and the c.i.c. (an abbreviation for ``completely integrally closed'') properties do not localize  (see Section \ref{examples} below).

When studying the Archimedean property, we  use Corollary \ref{fsohm}: a  stable domain $R$ is  Archimedean if and only if each nonunit of $R$ belongs to a height-one prime ideal of $R'$ (this result is related to Ohm's Theorem for Pr\"ufer domains \cite[Corollary 1.2]{Ohm}). We also prove that a stable domain is locally Archimedean if and only if $\bigcap_{n\ge1} M^n=(0)$ for each maximal ideal $M$ (Proposition \ref{larchd}); this condition implies accp (Proposition \ref{accpcond}). Hence a stable locally Archimedean domain satisfies accp (Corollary \ref{stableaccp}).

By Example \ref{ex3},  a stable Archimedean  domain need not be locally Archimedean, and by Example \ref{ex2} a semilocal Archimedean domain (even completely integrally closed) need not be locally Archimedean. On the positive side we show that an Archimedean finitely stable semilocal domain with stable maximal ideals is locally Archimedean (Proposition \ref{semiarch}).

The following results (1.1-1.4), due to Olberding, are basic for our paper.
\begin{theorem}\label{OlbCharFS2}\cite[Corollary 5.11]{OG} 
A domain $R$ is finitely stable if and only if it satisfies the following conditions:
\begin{enumerate} 
\item 
$R'$ is a quadratic extension of $R$;
\item
$R'$ is a Pr\"ufer domain;
\item Each maximal ideal of $R$ has  at most $2$ maximal ideals of $R'$ lying over it.
\end{enumerate}
\end{theorem}
 	
Recall that a domain $D$ is a {\em quadratic extension} of a domain $R$ if for each $x,y\in D$ we have $xy \in xR+yR+R$.   
Olberding  also proved that, in the local one-dimensional case, stability and finite stability are equivalent provided the maximal ideal is stable:

\begin{proposition}  \label{Olb1dims2} \cite[Theorem 4.2]{O8} Let $R$ be a local one-dimensional domain. The following conditions are equivalent:
\begin{enumerate}
\item[(i)] $R$ is stable;
\item[(ii)] $R$ is finitely stable  with stable maximal ideal;
\item[(iii)] $R'$ is a quadratic extension of $R$ and $R'$ is a Dedekind domain with at most two maximal ideals.
\end{enumerate}
\end{proposition}

\begin{construction}\label{olbcons}\cite[Section 4]{O2}
	\rm Let $(R,M)$ be a local  domain.
	Set $R_i=\{0\}$ for $i<0$, $R_0=R$ and $M_0=M$. Define inductively for $n > 0$:
	\begin{align*}
R_n&= R_{n-1} \text{ if } R_{n-1} \text{ is not local },\\ 
\text{ and } R_n&=E(M_{n-1})=(M_{n-1}:M_{n-1}) \text{ if }
	R_{n-1} \text{ is local with maximal ideal denoted by } M_{n-1}.
\end{align*}
	Set $T=\bigcup_{n\geq 0}R_n$.
	
	Thus we have:
	\begin{enumerate}
		\item[(a)]
		If there exists an integer $k>0$ such that
		$R_k$ is not local, but $R_i$ is local for $0\le i<k$, then $R_n=R_k$ for all $n\ge k$, and $T=R_k$.
		\item[(b)]
		If $R_n\subsetneq R_{n+1}$ for all $n\ge0$, all the rings $R_n$ are local.
	\end{enumerate}
\end{construction}

We will use  repeatedly the following  theorem of Olberding.	

\begin{theorem} \cite[Corollary 4.3, Theorem 4.8]{O2} {\rm and its proof, and} \cite[Theorem 5.4]{OG} \label{Olb2} Let $R$ be a finitely stable local domain with stable maximal ideal $M$.  With the notation of \ref{olbcons} we have:

\begin{enumerate}
		\item[(1)] Each $R_n$ is finitely stable with stable maximal ideals, and there exists an element $m\in M$ such that $M=mR_1$. Moreover, for $k\geq 1$, if $R_k$ is local with maximal ideal $M_k$, then  $M_k=mR_{k+1}=MR_{k+1}$, and if $T$ is local, then its maximal ideal is $mT=MT$.   

\item[(2)] Each $R_n$ is a finitely generated $R$-module, thus $T$ is an integral extension of $R$. 
\end{enumerate}

We also have:
	
\begin{enumerate}
\item[(a)]  If  $T=R_n$ for some $n\geq 0$, then $T$ is a finitely generated $R$-module, and $T$ has at most two maximal ideals.
		
\item[(b)] If $T\neq R_n$ for all $n\geq 0$,  then  $T$ is local.

\item[(c)] The maximal ideals of $T$ are principal, and the Jacobson radical of $T$ is equal to
$mT=MT$, where $mR_1=M$.
\end{enumerate}
In addition,
 if $R$ is  a  stable domain, then $T$ is equal to the integral closure $R^\prime$ of $R$, and $R'$ is a strongly discrete Pr\"ufer domain.
 	\end{theorem}

\section{On the Archimedean  property}

We start with some generalities on the Archimedean property. Then we prove that 
a finitely stable domain $R$ with stable maximal ideals is locally Archimedean if and only if $\bigcap_{n\ge1} M^n=(0)$ for each maximal ideal $M$ of $R$ (Proposition \ref{larchd}). We deduce from this result that
a locally Archimedean stable domain satisfies accp (Corollary \ref{stableaccp}). 

Many results in this section are related to  the following theorem of J. Ohm, which will be extended in Theorem \ref{afsohm} below.

\begin{theorem}\label{ohm}\cite[Corollary 1.6]{Ohm}.
	Let $R$ be a Pr\"ufer domain. We have:
	
	{\rm (1)} If $a$ is a nonunit of $R$ belonging to just finitely many maximal ideals, then	$\bigcap_{n\ge1}a^nR=(0)$ if and only if  $a$ belongs to a height-one prime ideal.
	
	Hence:
	
	{\rm (2)}  If $R$ has  finite character, then
	$R$  is Archimedean  if and only if each nonunit of $R$ belongs to a height-one prime ideal.
\end{theorem}

\begin{corollary}\label{prufsemi}
	An Archimedean Pr\"ufer  domain of finite character and with just finitely many height-one prime ideals is one-dimensional.
	In particular, an Archimedean Pr\"ufer semilocal domain is one-dimensional.	 
\end{corollary}	  

\begin{proof} 
	Let $M$ be a maximal ideal of $R$. By Ohm's Theorem \ref{ohm} (2), $M$ is contained in the finite union of the height-one prime ideals of $R$, so it is contained in one of them by prime avoidance. Hence $M$ has height one, so $R$ is one-dimensional.
	
	If $R$ is Pr\"ufer and semilocal, then $R$ has just finitely many height-one prime ideals. Hence, if $R$ is Archimedean, then $R$ is one-dimensional.
\end{proof}	  

\begin{remark}\label{genarch} 
	An integral domain $R$ is Archimedean if and only if for each nonzero nonunit $r$ of $R$ there is an Archimedean domain $D$ (depending on $r$) containing $R$ such that $r$ is a nonunit in $D$. 
	Moreover, replacing $D$ by $D\cap \Frac(R)$, we  may assume that $D$ is an overring of $R$.
	
	In particular, an intersection of Archimedean domains is Archimedean. Hence a locally Archimedean domain is Archimedean.	 
\end{remark}

\begin{corollary}\label{easyohm}
Let $A \subseteq B$ be an extension of domains. If every nonzero nonunit  of $A$ belongs to a height-one prime ideal of $B$, then $A$ is Archimedean.	 
\end{corollary}	  

\begin{proof} 
Let $a$ be a nonzero nonunit of $A$. If $Q$ is an height-one prime ideal of $B$ containing $a$, then $a$ is a nonunit in the one-dimensional (so Archimedean) domain $B_Q$. By Remark \ref{genarch}, $A$ is Archimedean. 
\end{proof}	  

\begin{corollary}\label{OneArch} 
	Let $(R,M)$ be a local domain. If some integral extension of $R$ has a height-one maximal ideal, then $R$ is Archimedean.	 
\end{corollary}	  
\begin{proof} 
If $Q$ is a height-one maximal ideal of an integral extension $D$ of $R$, then $Q\cap R=M$. Hence $R$ is Archimedean by Corollary \ref{easyohm}. \end{proof} 

\begin{proposition}\label{archloc}
	Let $R$ be an integral domain, and let $a$ be a nonunit  of $R$ that belongs to just finitely many maximal ideals.
	
	Then 
	$\bigcap_{n\ge1}a^nR=(0)$ if and only if  $a$ belongs to a maximal ideal $M$ such that  $\bigcap_{n\ge1}a^nR_M=(0)$.
\end{proposition}

\begin{proof}
	Let $\mathfrak F$ be the set of maximal ideals containing $a$.
	We have
	$$\bigcap_{n\ge1}a^nR=R\cap\bigcap_{n\ge1}\bigcap_{M\in\mathfrak F}a^nR_M=\bigcap_{M\in\mathfrak F}\big(R\cap\bigcap_{n\ge1}a^nR_M\bigr).$$
	Since the set $\mathfrak F$ is finite it follows that
	$\bigcap_{n\ge1}a^nR=(0)$ if and only if $R\cap ~\bigcap_{n\ge1}(a^nR_M)=(0)$ for some $M\in \mathfrak F$, equivalently 
$ \bigcap_{n\ge1}a^nR_M=(0)$ for some $M\in \mathfrak F$.
\end{proof}

\begin{proposition}  \label{archDVR} If $R$ is an Archimedean domain and $P$ is a principal prime ideal of $R$, then $R_P$ is a DVR.
\end{proposition}
\begin{proof}
If $P=rR$, then by \cite[Theorem 7.6 (a) and (c)]{gilmer}, $\bigcap_{n\geq 0}P^n= \bigcap_{n\geq 0}r^nR=(0)$ is the largest prime ideal of $R$ properly contained in $P$. It follows that $R_P$ is a one-dimensional local domain with principal maximal ideal, and so $R_P$ is a DVR.
\end{proof}

\begin{corollary}\label{archPID} Let $R$ be an integral domain.
\begin{enumerate}
\item
If $R$ is Archimedean with principal maximal ideals, then $R$ is a principal ideal domain.

\item
If $R$ is locally Archimedean with invertible maximal ideals, then $R$ is a Dedekind domain.
\end{enumerate}
\end{corollary}

\begin{proof}
(1)
By Proposition \ref{archDVR}, $R$ is one-dimensional. Since every nonzero prime ideal of $R$ is principal, $R$ is a principal ideal domain by \cite[Corollary 37.9]{gilmer}.

(2)
By Proposition \ref{archDVR}, $R$ is locally a DVR (i.e., $R$ is almost Dedekind); in particular $R$ is one-dimensional. It follows that $R$ is a Dedekind domain by \cite[Theorem 37.8 $(1)\lra (4)$]{gilmer}.
\end{proof}

However, an Archimedean  domain $R$ with invertible maximal ideals is not necessarily one-dimensional, even if $R$ is  Pr\"ufer and stable: see Example \ref{ex3} below.

\begin{corollary}\label{bezout}
	An Archimedean B\'ezout domain $R$ with stable maximal ideals is  a principal ideal domain.
\end{corollary}

\begin{proof}
	As mentioned at the end of the proof of  \cite[Lemma 4.5]{O2}, a stable maximal ideal $M$ of a Pr\"ufer domain $R$ is invertible since \mbox{$(M:M)=R$}. Thus the maximal ideals of $R$ are finitely generated, so they are principal. Hence $R$ is a principal ideal domain by Corollary \ref{archPID}.
\end{proof}

None of the two conditions on the B\'ezout domain $R$ in Corollary \ref{bezout} to be a principal ideal domain can be omitted. Indeed, $R=\mathbb Z+X\mathbb Q[X]$ is a two-dimensional  B\'ezout domain with principal maximal ideals. On the other hand, the ring of entire functions is an   infinite-dimensional completely integrally closed  B\'ezout domain. Thus $R$ is Archimedean; see also Remark \ref{bezcic} below. Hence  $R$ has non-principal maximal ideals: these are the free maximal ideals: see \cite[Ch. VIII, \S 8.1]{fhp} and \cite[Ch.6, \S3]{Re}.

\begin{remark}\label{bezcic} 
By \cite[Corollary 3.1]{BB}, a GCD domain (in particular, a B\'ezout domain) is Archimedean if and only if it is completely integrally closed. 
\end{remark}	  

\begin{lemma}
	\label{I2xI}  \cite[Corollary 5.7]{OG}
	Let $R$ be a  finitely stable local domain. 
	Then a stable ideal $I$ of $R$ is principal in $(I:I)$. Moreover, if $I=x(I:I)$, then $I^2=xI$.
\end{lemma}

\begin{proposition}\label{larchd}
	Let $R$ be a finitely stable domain with stable maximal ideals. Then
	$R$ is locally Archimedean if and only if $\bigcap_{n\ge1} M^n=(0)$ for each maximal ideal $M$.
\end{proposition}

\begin{proof}
	Let $M$ be a maximal ideal of $R$. By Lemma \ref{I2xI}, there exists an element $m\in M$ such that $M^2=mM$.  Hence
$\bigcap_{n\ge0}M^nR_M = \bigcap_{n\ge0}m^nR_M$. It follows that $R_M$ is Archimedean if and only if $\bigcap_{n\ge0}M^nR_M =(0)$.	
We have	
	 	$$
	\bigcap_{n\ge1} M ^n=\bigcap_{n\ge1}(M^n R_M\cap R)=\left(\bigcap_{n\ge1}M^n R_M\right)\cap R,
	$$	 
	so $R_M$ is Archimedean if and only if $\bigcap_{n\ge1} M^n=(0)$. The proposition follows.
\end{proof}

\begin{proposition}\label{accpcond} 
	Let $R$ be an integral domain of finite character such that $\bigcap_{n\ge1}M^n=(0)$ for each maximal ideal $M$ of $R$. Then $R$ satisfies accp.	 
\end{proposition}	  

\begin{proof} 
	Assume that $R$ does not satisfy accp. Then there exists an infinite   sequence of nonunits $r_n$ in $R$  such that $\bigcap_{n\ge1}(\prod_{i=1}^nr_iR)\ne(0)$. Let $c$ be an  element in this intersection.   For all $n\geq 1$,  each maximal ideal containing $r_n$ contains also $c$, since $c\in r_nR$. As $c$ belongs to just finitely many maximal ideals, there exists a maximal ideal $M$ containing $c$ such that $r_n\in M$ for infinitely many  $n$'s. Thus for each $n\ge1$, there exist integers $1~\le~i_1~<~i_2~<~\dots~<i_n$ such that
	$r_{i_k}\in M$ for all $1\le k\le n$. We have $c\in \prod_{j=1}^{i_n} r_jR\subseteq M^n$.
	Hence $c\in \bigcap_{n\ge1}M^n$, a contradiction.
\end{proof}

From Proposition \ref{accpcond} we obtain, by  using Proposition \ref{larchd}:

\begin{corollary}\label{stableaccp} 
A locally Archimedean finitely stable domain with stable maximal ideals and of finite character (in particular, a locally Archimedean stable domain) satisfies accp.
\end{corollary}

However a domain $R$ of finite character satisfying accp is not necessarily locally Archimedean, even if $R$ is stable (see Example \ref{ex3} below).

\section{An extension of Ohm's Theorem to finitely stable domains}

We extend Ohm's Theorem  from Pr\"ufer domains to finitely stable domains  (Theorem \ref{afsohm}). We present a criterion for the locally Archimedean property of a stable domain in  Proposition \ref{fsLocOhm}.
  As an application, we prove that a semilocal finitely stable Archimedean domain is locally Archimedean (Proposition \ref{semiarch}). 

\begin{proposition}\label{AB} 
	Let $A$ be a finitely stable domain,  let  $a$ be a nonzero element of $A$, and let $B$ be an integral extension overring of $A$. Then:
	$$
	\bigcap_{n\ge1}a^n A=(0)\lra \bigcap_{n\ge1}a^nB=(0).
	$$
	\end{proposition}

\begin{proof} 
Assume that $\bigcap_{n\ge1}a^nB\ne (0)$. Since $B$ is an overring of $A$, there exists a nonzero element $c$ in $ A\cap\bigcap_{n\ge1}a^nB$.
Thus for all $n\ge1$ we have $c=a^nb_n$, where $b_n\in B$. Since  by Theorem \ref{OlbCharFS2}, $B$ is a quadratic extension of $A$, we have for all $n$: $b_n^2\in b_nA+A$. It follows that  $a^nb_n^2\in A$. Hence $c^2=a^n(a^nb_n^2)\in a^nA$ for all $n\ge1$. Thus $\bigcap_{n\ge1}a^n A\ne(0)$, and the proposition follows.
\end{proof}	  

\begin{proposition}\label{archDN} 
	Let $(R,M)$ be a finitely stable local domain with stable maximal ideal, and let $D$ be an integral extension overring of $R$.
	
	Then $R$ is Archimedean if and only if $D$ has a maximal ideal  $N$ such that $D_N$ is Archimedean.
\end{proposition}	  

\begin{proof} 
	Assume that $R$ is Archimedean. By  Lemma \ref{I2xI}  we have $M^2=mM$ for some element  $m\in M$.
	 By Proposition \ref{AB}, $\bigcap_{n\ge1}m^nD=(0)$. By Theorem \ref{OlbCharFS2}, $D$ has at most two maximal ideals.	 By Proposition \ref{archloc},  there exists a maximal ideal $N$ of $D$ such that $\bigcap_{n\ge1}m^nD_N=(0)$. Since $D$ is an integral extension of $R$ and $R$ is local, it follows that a prime ideal of $D$ contains $M$ if and only if it is  a maximal ideal of $D$. Hence the only prime ideal of $D_N$ containing $MD_N$ is $ND_N$, so
	$ND_N=\sqrt {mD_N}$. If $x\in ND_N$, then $x^k\in mD_N$ for some integer $k\ge1$. 
	Thus $\bigcap_{n=1}x^nD_N\subseteq \bigcap_{n=1}(x^k)^nD_N\sub \bigcap_{n=1}m^nD_N =(0)$, so  $D_N$ is Archimedean.
	
Conversely, if $D_N$ is Archimedean, then $R$ is Archimedean by
Remark \ref{genarch} since $R\subseteq D_N$ and $N\cap R=M$. 
	\end{proof}	

\begin{proposition}\label{fslArch} 
	Let $(R,M)$ be a   local domain. 
	\begin{enumerate} 
		\item 
		If some integral extension of $R$ has a height-one   maximal ideal, then $R$ is Archimedean.   
		\item 
		Conversely, we have:
\begin{enumerate} 
\item 
If $R$ is Archimedean and finitely stable, then  $R'$  has a height-one maximal ideal. 

\item
If $R$ is Archimedean, finitely stable and the ideal $M$ is stable, then  $T$ has a height-one maximal ideal ($T$ is defined in  Construction \ref{olbcons}).
\end{enumerate}
\end{enumerate}	 
\end{proposition}

\begin{proof} 
(1)  is Corollary \ref{OneArch}.
 
 (2, a)
	By Theorem \ref{OlbCharFS}, $R'$ has at most two maximal ideals. Since $R'$ is Pr\"ufer, $R'$ has at most two height-one prime ideals: $Q_1$ and $Q_2$ (not necessarily distinct). Let $P_i=Q_i\cap R$, $i=1, 2$. Since $R$ is Archimedean, by Proposition \ref{AB} we have $\bigcap_{n\ge1}a^nR'=(0)$ for all $a\in R$. By Theorem \ref{ohm}, $M \subseteq P_1\cup P_2$. We may assume that $M \subseteq P_1$, so $M=P_1=Q_1\cap R$. Hence $Q_1$ is a height-one  maximal ideal of $R'$.
	
(2, b) 
By Proposition \ref{archDN}, $T$ has a maximal ideal  $N$   such that the domain $T_N$ is Archimedean. Hence $T_N$ is  a DVR by Proposition \ref{archDVR} as $N$ is a principal ideal by Theorem \ref{Olb2} (c). Thus  $N$ is a height-one maximal ideal of $T$.
\end{proof}	

In the notation of \ref{olbcons}, if $R_k$ is one-dimensional for some $k\ge0$, then all the rings $R_n$, as well as $T$, are one-dimensional since $T$ is an integral extension of $R_n$, for all $n\ge0$. For the Archimedean property we have:

\begin{corollary}\label{alln} 
	Let $(R,M)$ be a finitely stable local domain with stable maximal ideal. Set $R_{\infty}=T=\bigcup_{n\geq 0}R_n$ (see Construction \ref{olbcons}). Assume that  $R_k$  is Archimedean for some $0\le k\le\infty$. Then $R_n$ is Archimedean for each $n$ such that $R_n$ is local. Thus $R_n$ is Archimedean at least for each $R_n\ne T$. 
\end{corollary}	  

\begin{proof} 
	For all $0\le n\le\infty$ we have $(R_n)'=R'$, so the corollary follows from Proposition \ref{fslArch}.	  
\end{proof}	  

Corollary \ref{alln} might fail when $T$ is not local, so $T=R_n$ for some integer $n$. Indeed, in Example \ref{ex6}, $R$ is a stable local Archimedean domain, but $T=R'=R_1$ is not Archimedean. See also Proposition \ref{Tarch} below.

\begin{proposition}\label{fsLocOhm} 
	Let $R$ be a finitely stable  domain. The following  conditions are equivalent:
	\begin{enumerate} 
		\item[\rm(i)] 
		$R$ is locally Archimedean; 
		
		\item[\rm(ii)] 
		Each  maximal ideal of  $R$ is contained in a height-one prime ideal of $R'$ (which is necessarily maximal);
	\end{enumerate}
\end{proposition}

\begin{proof} 
(i) $\ra$ (ii)
If $R$ is local, then  (ii) follows from Proposition \ref{fslArch}(2)(a).

In the general case, let $M$ be a maximal ideal of $R$. By the local case, the ideal $MR_M$ of $R_M$ is contained in a height-one prime $Q$ of $(R_M)'=R'_M$, where $R'_M$ is the localization of $R'$ at the multiplicative subset $R\setminus M$. Thus $Q\cap R'$ is a height-one prime ideal of $R'$ containing $M$.	 

(ii) $\ra$ (i)
Let $M$ be a maximal ideal of $R$. Let $Q$ be a height-one prime ideal of $R'$ containing $M$. Thus $QR'_M$ is a height-one prime ideal of  $R'_M=(R_M)'$ containing $M$. By Corollary \ref{OneArch}, $R_M$ is Archimedean, so $R$ is locally Archimedean.
\end{proof}

\begin{proposition}\label{Tarch} 
Let $R$ be a finitely stable local domain with stable maximal ideal. In the notation of \ref{olbcons}, $T$ is Archimedean if and only if $R$ is one-dimensional. 	\end{proposition}	  

\begin{proof} 
If $R$ is one-dimensional, then $T$ is one-dimensional, and so Archi\-me\-dean, since $T$ is an integral extension of $R$.
Conversely, if $T$ is Archimedean, then $T$, and so also $R$, is one-dimensional by Corollary \ref{archPID},
as the maximal ideals of $T$ are principal. 
\end{proof}	  

\begin{corollary}\label{corTarch}  \cite[Theorem 2.8]{GR1} Let $R$ be a finitely stable local domain with stable maximal ideal. If $R$ is Archimedean and $T$ is local, then $R$ is one-dimensional.
\end{corollary}
\begin{proof} 
If $T$ is local, then $T$ is Archimedean by Corollary \ref{alln} and so $R$ is one-dimensional by Proposition \ref{Tarch}.
\end{proof}

We now state the promised generalization of Ohm's Theorem \ref{ohm}. 

\begin{theorem}\label{afsohm}
	Let $R$ be a finitely stable  domain, and let $a$ be a nonzero nonunit of $R$ belonging to just finitely many maximal ideals of $R$. The following conditions are equivalent:
	\begin{enumerate} 
		\item[\rm(i)] 
		$\bigcap_{n\ge1}a^nR=(0)$;
		
		\item[\rm(ii)]  
		$a$ belongs to a height-one prime ideal of $R'$;
		
		\item[\rm(iii)] 
		$a$ belongs to a  prime ideal $P$ of $R$ such that the domain $R_P$ is Archimedean.
	\end{enumerate} 
\end{theorem}	  	  

\begin{proof} 
	(i) $\ra$ (ii)
	By Proposition \ref{AB},  $\bigcap_{n\ge1} a^nR'=(0)$. If $N$ is a maximal ideal of $R'$ containing $a$, then $N\cap R$ is  a  maximal ideal of $R$ containing $a$. Since each maximal ideal of $R$ is contained in at most two  maximal ideals of $R'$ (Theorem \ref{OlbCharFS}), it  follows that $a$ belongs to just finitely many maximal ideals of $R'$. Since $R'$ is Pr\"ufer,   $a$ belongs to a height-one prime ideal of $R'$ (Theorem \ref{ohm}).	 
	
	(ii) $\ra$ (iii)
Let $Q$ be a height-one prime ideal of $R'$ containing $a$, and let $P=Q\cap R$.
By Corollary \ref{easyohm} for $A=R_P$ and $B=R'_Q$, we obtain that $R_P$ is Archimedean.
	
	(iii) $\ra$ (i)
follows from Remark \ref{genarch}.
\end{proof}	  

\begin{corollary}\label{fsohm} 
Let  $R$ be a finitely stable domain of finite character (this holds, in particular, if $R$ is a stable domain). Then $R$ is Archimedean if and only if  every nonzero nonunit in $ R$ belongs to a height-one prime ideal of $R'$.
\end{corollary}	  

In the next proposition we extend Corollary \ref{prufsemi} to finitely stable domains:

\begin{proposition}\label{semiarch}
	An Archimedean finitely stable domain  of finite character such that its integral closure has just finitely many height-one prime ideals is locally Archimedean.
	In particular, an Archimedean finitely stable semilocal domain is locally Archimedean.	 
\end{proposition}	  

\begin{proof} 
Let $M$ be a maximal ideal of $R$. As $R$ is Archimedean, by Theorem \ref{afsohm},  $M$ is contained in the finite union of the height-one primes of $R'$. Thus the ideal $MR'$ of $R'$ is contained in one of these primes. By Proposition \ref{fsLocOhm},  $R$ is locally Archimedean.
	
	If $R$ is an Archimedean  finitely stable semilocal domain, then $R'$ is Pr\"ufer and semilocal. Thus $R'$ has just finitely many height-one prime ideals. It follows that $R$ is locally Archimedean.
\end{proof}	  

In connection with Proposition \ref{semiarch}, 
by Example \ref{ex3},  a stable Archi\-me\-dean  domain need not be locally Archimedean, and by Example \ref{ex2} a semilocal Archimedean (even completely integrally closed) domain need not be locally Archimedean.

\begin{question} 
By Proposition \ref{fsLocOhm}, if a finitely stable  domain $R$ is  locally Archimedean, then each nonzero nonunit of $R$ belongs to a height-one maximal ideal of $R'$. Is the converse true? Cf. Corollary \ref{fsohm}.	 
\end{question}	  

\section{One-dimensionality of Archimedean stable domains}

In \cite[Theorem 4.8]{GR1}, we proved that a finitely stable Mori domain is one-dimensional.
 In this section, we illustrate a general method for constructing local  Archimedean stable domains  of any dimension (Propositions \ref{V1V2} and  \ref{intV1V2}); see also Example \ref{ex6} below. 

First we state a criterion for one-dimensionality of an Archimedean  stable domain. We say that a domain $R$ is \emph{equidimensional} if $\dim R=\dim R_M$, for each maximal ideal $M$. 

\begin{proposition}\label{equidim} 
	Let $R$ be an Archimedean finitely stable  domain of finite character (this includes the case that $R$ is  Archimedean and stable). The following  conditions are equivalent:
	\begin{enumerate} 
		\item[(i)] 
		$R$ is one-dimensional; 
		
		\item[(ii)] 
		Every integral extension of $R$ is equidimensional;
		
		\item[(iii)] 
		$R'$ is equidimensional;
		
		\item[(iv)]
		The pair $(R,R')$ satisfies GD (the going down property) and $R$  is equidimensional.
	\end{enumerate} 
	\end{proposition}	  

\begin{proof} 
	(i) $\ra$ (ii)
	Every integral extension of $R$ is one-dimensional, so also equidimensional.

	(ii) $\ra$ (iii) Obvious.
	
	(iii) $\ra$ (i)
	By Corollary \ref{fsohm}, $R'$ has a height-one maximal ideal. Thus $R'$ is one-dimensional, and so is $R$.
	
	(i) $\ra$ (iv) Clear.
	
	(iv) $\ra$ (iii) Indeed, if $B$ is any ring extension of an equidimensional (in particular, local) ring $A$ such that the pair $(A,B)$ satisfies GD, then $B$ is equidimensional and $\dim B=\dim A$.
\end{proof}
\begin{proposition} 
Let $R$ be an Archimedean finitely stable local domain with stable maximal ideal. The following conditions are equivalent:

\begin{enumerate} 
\item[\rm(i)] 
$R$ is one-dimensional;

\item[\rm(ii)] 
$R$ is Mori;
 
 \item[\rm(iii)]
$T$ is Archimedean;

\item[\rm(iv)]
$T$ is equidimensional;

\item[\rm(v)]
The pair $(R,T)$ satisfies GD.
\end{enumerate}
	(See \ref{olbcons} for the notation $T$.)
\end{proposition}
 	
\begin{proof} 

(i) $\ra$ (ii) $R$ is stable by Proposition \ref{Olb1dims}. Thus $R$ is Mori by \cite[Proposition 3.13]{GR1}.

(ii) $\ra$ (i) by \cite[Proposition 4.7]{GR1}.

(i) $\lra$ (iii) is Proposition \ref{Tarch}.

(i) $\ra$ (iv) because $T$ is one-dimensional. 
 
(iv) $\ra$ (i) 
$T$ has a height-one maximal ideal by Proposition \ref{fslArch} (2,b). Thus $T$ is one-dimensional, and so is $R$.

(i) $\ra$ (v) 
This follows from that both $R$ and $T$ are one-dimensional.

(v) $\ra$ (iv) Since $R$ is local, we may use the proof of the implication (iv) $\ra$ (iii) in Theorem \ref{equidim}. 
\end{proof}

\begin{proposition}\label{GD} 
	Let $R$ be an Archimedean finitely stable  semilocal domain. Then $R$ is one-dimensional if and only if the pair  $(R,R')$ satisfies GD.
\end{proposition} 
 
\begin{proof} 
Assume that the pair  $(R,R')$ satisfies GD. Let $M$ be a maximal ideal of $R$. By Corollary \ref{fsohm}, $M$ is contained in the union of the height-one maximal ideals of $R'$. Since $R'$ is semilocal (Theorem \ref{OlbCharFS}), it follows that $M$ is contained in a height-one maximal ideal $N$ of $R'$.
	As the pair $(R,R')$ satisfies GD, this implies that $M$ has height one, so $R$ is one-dimensional.   	 
\end{proof}	  

\medskip
We now turn to the question how to obtain an Archimedean stable local domain $(R,M)$ of dimension greater than one.  Here we use again  Olberding's work, and also a useful suggestion of W. Heinzer. 

If $R$ is such a domain, with the usual notation \ref{olbcons}, by  Corollary \ref{corTarch}, $T$ is not local and so $R$ must satisfy condition (a) of Theorem \ref{Olb2}, that is, $T= R_n$ for some $n\ge0$. Since $R$ is stable,
$T=R'$ is a Pr\"ufer domain and $T$ has exactly $2$ maximal ideals, which we denote by $N_1$ and $N_2$.
Since $R$ is Archimedean, $T$ has a height-one maximal ideal by Proposition \ref{fslArch}. We  may assume that $\height N_1=1$ and  $\height N_2>1$. 
Let $T=R_k$ with minimal $k\ge0$, so $k>0$ since $T$ is not local. Thus $R_{k-1}$ is local and, since any overring of a stable domain is stable \cite[Theorem 5.1]{O2}, $R_{k-1}$ is stable. By 
Corollary \ref{alln}, $R_{k-1}$ is Archimedean. Also, $\dim R_{k-1}=\dim R>1$. Replacing $R$ by $R_{k-1}$, we may assume that $R_1=T$. 

We have canonical isomorphisms $R/M\cong T/N_i$ for $i=1,2$, so $T=R+N_1=R+N_2$. In Example \ref{ex6}, $k$ is a subfield of $R$ canonically identified with $R/M$, so $T=k+N_1=k+N_2$ and $M=N_1\cap N_2$.

We use the following lemma of Olberding:
\begin{lemma}\label{stableRR} Let $R$ be a finitely stable domain. If $I$ is a nonzero ideal of $R$ such that $IR'$ is principal, then $I$ is principal in $(I:I)$, in particular $I$ is stable.
\end{lemma}	  
\begin{proof} By Theorem \ref{OlbCharFS}, $R'$ is a quadratic extension of $R$ and has at most two maximal ideals. Hence we can apply \cite[Proposition 3.6]{OG}.
\end{proof}

For the next proposition cf. \cite[Theorem 14]{J}. (The statement in the proof of \cite[Theorem 14]{J}, that $u-u^2\in R$ for each nonunit $u\in A$, is false in general, but this error can be easily corrected.)

Recall that a Pr\"ufer domain is strongly discrete if its nonzero prime ideals are not  idempotent.

\begin{proposition}\label{V1V2} 
Let $(V_1, Q_1)$ and $(V_2, Q_2)$ be two  valuation domains with no inclusion relation among them, with principal maximal ideals, with the same field of fractions $L$,  containing a field $k$, and such that $V_i=k+Q_i$, for $i=1,2$.  Let $D=V_1\cap V_2$, $N_i=Q_i\cap D$, for $i=1,2$, $M=N_1\cap N_2$, $R=k+M$, and $R_1=(M:M)$.
Then:

\begin{enumerate} 
\item
$N_1$,  $N_2$ are the only maximal ideals of $D$,  $N_1\ne N_2$, and $N_1$,  $N_2$ are principal. We have $D_{N_i}=V_i$ for $i=1,2$, so $D$ is Pr\"ufer, $\Frac D=L$, and  $D=k+N_i$ for $i=1,2$.

\item
If  $N_1=xD$, then $D=R+xR$, so $D$ is a $2$-generated $R$-module. 
Moreover, $M$ is a principal ideal of $D$ and a $2$-generated ideal of $R$.
Also $D$ is a quadratic extension of $R$ and $D=R'=R_1$.

\item 
$R=k+M$ is a  local domain with maximal ideal $M$.

\item 
The domain $R$ is finitely stable with stable maximal ideal $M$.

\item
$R$ is stable if and only $D$ is stable, equivalently $D$ is strongly discrete.

\item
$R$ is Archimedean if and only if one of the two valuation domains $V_1,V_2$ is one-dimensional, and so a DVR.

\item
$\dim R>1$ if and only if $\dim V_i>1$ for some $i=1,2$.
\end{enumerate}
\end{proposition}	  

\begin{proof} 
(1)
 By \cite[Theorem 12.2]{Mats}, $N_1$ and  $N_2$ are the only maximal ideals of $D$, $N_1\ne N_2$, and   $D_{N_i}=V_i$ for $i=1,2$. For $i=1,2$, the maximal ideal $N_i$ of  $D$ is locally principal, and so it is principal since $D$ is semilocal. 

For $i=1,2$ we have natural isomorphisms $D/N_i\cong D_{N_i}/N_i D_{N_i}=V_i/Q_i\cong k$, implying that $D=k+N_i$.

(2)
 Since the ideals $N_1,N_2$ of $D$ are principal, we deduce that  also $M=N_1N_2$ is a principal ideal of $D$. Thus $D$ is an overring of $R$, and $D=(M~:~M)=R_1$.
 
 Since $xN_2\subseteq N_1N_2=M$, we have:
 $$
 D=k+N_1=k+xD=k+x(k+N_2)=k+xN_2+xk \subseteq R+kx.
 $$
 Hence $D=R+xR$ is a finite $R$-module, so $D$ is an integral extension of $R$.
 Since $D$ is a Pr\"ufer domain, it follows that $D=R'$. 
 
 As $M$ is a principal ideal of $D$ and $D=R+Rx$ is a $2$-generated $R$-module, it follows that $M$ is a $2$-generated ideal of $R$.
 
(3)
By definition, $R=k+M$, so $M$ is a maximal ideal of $R$. If $P$ is a maximal ideal of $R$ then $P= N_i\cap R$ for some integer $i=1,2$, because $D=R'$ by (2). Thus $M=(N_1\cap R)\cap (N_2\cap R)\subseteq  P$, implying that $M=P$. Thus $(R,M)$ is a local domain.

(4)
By item (1), $D=R'$ is Pr\"ufer with two maximal ideals and by item (2), $D$ is a quadratic extension of $R$. By Olberding's characterization \ref{OlbCharFS}, $R$ is finitely stable. Since $M$ is a principal, so stable, ideal of $D$ and
$\Frac R=\Frac D$, it follows that $M$ is a stable ideal of $R$. 

(5)
If $R$ is stable then $D$ is stable, since each overring of a stable domain is stable.

Conversely, assume that $D$ is stable. By item (2), we have $D=R'=R_1$ and $M=mR'$, where $m\in M$.

 Let $I$ be a nonzero ideal of $R$, and let $A=(I:I)$. 
 
 The domain  $D$ is Pr\"ufer, and, as shown at the end of the proof of \cite[Theorem 4.2]{O2}, $D=R_1$ is a minimal overring of $R$.  Hence by \cite[Proposition 2.4 and Terminology on page 137]{GH},  $D$ is contained in every overring of $R$ that is different from $R$. Hence, either $A=R$, or $D\subseteq A$.
 If $D\subseteq A$, then $A=(I:I)$ is a stable domain, so the ideal $I$ of $A$ is invertible in $(I:I)$, implying that $I$ is a stable ideal of $R$.
 
 Now assume that $A=R$. 
 Since $M$ is a principal ideal of $R'$, it follows that $ (IR':IR')=(IM:IM)$. Also $(IM:IM)MI\subseteq IM\subseteq I$, so  $(IM:IM)M\subseteq (I:I)=R$. Hence $(IR':IR')\subseteq  (R:M)$. If $(R:M)\ne(M:M)$, then the maximal ideal $M$ of the local domain $R$ is invertible, so principal, implying that $R= R_1=(M:M)$, a contradiction.
 If  $(R:M)=(M:M)$, then  $(IR':IR')=R'$. Hence $IR'$ is an invertible, so principal, ideal of $R'$ since $R'$ is stable and semilocal. By  Lemma \ref{stableRR}, $I$ is a stable ideal.
 
(6)
Since $R$ is local finitely stable, and $D=R'$, this follows from Proposition \ref{fslArch}.

(7)
Indeed, $\dim R=\dim D=\max (\dim V_1,\dim V_2)$.
\end{proof}

\begin{corollary}\label{intV1V2} 
Let	  $(V_1, Q_1)$ and $(V_2, Q_2)$ be two  strongly discrete valuation domains with no inclusion relation among them, with principal maximal ideals, with the same field of fractions $L$,  containing a field $k$, and  such that $V_i=k+Q_i$, for $i=1, 2$. Suppose $\dim V_1=1$ and $\dim V_2=n$, where $2\le n\leq \infty$. Let $D=V_1\cap V_2$, $N_i=Q_i\cap D$, for $i=1, 2$, $M=N_1\cap N_2$, and $R=k+M$.

Then $R$  is an $n$-dimensional Archimedean stable local domain.
 
Moreover, we have:
\begin{enumerate} 
\item 
$R$ satisfies accp, but $R'$ is not  Archimedean.

\item
The pair $(R,D)$ does not satisfy GD (the going down property). 
\end{enumerate}
\end{corollary}	  

\begin{proof} 
By Proposition \ref{V1V2}, $R$ is an $n$-dimensional Archimedean  stable local domain since  $D$ is a strongly discrete Pr\"ufer domain.

(1) 
By Corollary \ref{stableaccp}, any Archimedean stable domain satisfies accp. By  Corollary \ref{prufsemi}, $R'$ is not Archimedean since $R'$ is semilocal of  dimension greater than $1$. 
 
(2)
 $R$ does not satisfy GD by Proposition \ref{GD}.
\end{proof} 	  

\section{Examples} \label{examples}
It is well-known that the accp and the Archimedean properties do not localize. In \cite[Example 2]{Gr} Anne Grams constructs a one-dimensional Pr\"ufer domain of finite character which satisfies  accp (the ascending chain condition on principal ideals) and each of its localizations but one is a DVR, while the  other one is a valuation domain that is not a DVR, so it does not satisfy accp (see comments and more examples in \cite{AAZ} and its references).
Also, \cite{Gr} (page 328) provides a general construction of an almost Dedekind domain $A$ with accp whose Nagata ring $A(X)$ is not an accp domain (so that $A[X]$ is accp, while its localization $A(X)$ is not accp). This example as well as \cite[Example 2]{Gr} is one-dimensional, so it is locally Archimedean.

The ring of entire functions $E$ is an infinite-dimensional completely integrally closed (hence Archi\-me\-dean) B\'ezout domain \cite[Section 8.1]{fhp}, but it is not locally Archimedean since the localizations at maximal ideals are valuation domains, and a valuation domain that is not a field is Archimedean if and only if it is one-dimensional. The ring $E$ does not satisfy accp and it does not have finite character: for example, if $f$ is a nonzero entire function with infinitely many zeros  $c_1,c_2,\dots$ (e.g., $\sin z$), then
$
f\in \bigcap_{n=1}^{\infty}\prod_{i=1}^n(Z-c_i),
$
so the domain $E$ does not satisfy accp, and $E$ does not have finite character since $f$ belongs to the maximal ideals $(Z-c_i)E$ for all $i$.  

We construct in Example \ref{ex2} below a completely integrally closed (for short, c.i.c.) domain $R$ satisfying accp with only 
two maximal ideals such that,  for each  maximal ideal $M$, $R_M$ is not Archimedean; thus $R_M$ does not satisfy accp. Of course, $R$ is Archimedean and has finite character. We construct first  a c.i.c. local domain  $A$ with accp such that $A_P$ is not Archimedean for some prime ideal $P$ (Example \ref{ex1}). Then  we ``double'' this construction to obtain Example \ref{ex2} (see Remark \ref{double}).

We also construct a stable Pr\"ufer domain $R$ with accp  that is not locally Archimedean (Example \ref{ex3}), thus the converse of Corollary \ref{stableaccp} is false.

In Example \ref{ex4} we give an example of a local one-dimensional domain $R$ such that $R'$ is a finite extension of $R$, the ring $R'$ is a PID, so stable, but $R$ is not even finitely stable (cf. Proposition \ref{V1V2} (5) and Lemma \ref{stableRR}).

Finally, following Olberding  (\cite[Proposition 5.4]{O3}), in Example \ref{ex5} we construct a stable valuation domain with prime spectrum consisting of an infinite descending chain of prime ideals. We use this example in the last Example \ref{ex6},
where we present a stable Archimedean local domain of arbitrary dimension.

\bigskip

Recall that a set of subrings $\mathbf S$ of a ring $R$ is {\em directed} if for each $A,B\in \mathbf S$ there exists $C\in\mathbf S$ such that both $A$ and $B$ are contained in $C$.

\begin{proposition}\label{noeth}
Let $R$ be an integral domain that is  a directed union of  a set  $\mathbf S$ of integrally closed Noetherian subrings.
Assume  that for every $A\in\mathbf S$  there exists a retraction $\varphi_A: R\to A$ mapping nonunits of $R$ to nonunits of $A$.
Then $R$  is c.i.c. and it satisfies accp.
\end{proposition}

\begin{proof}
Assume  that $R$ does not satisfy accp. Hence there exists a strictly increasing infinite sequence of nonzero principal ideals in $R$:
$$r_1R\subsetneq r_2R\subsetneq r_3R\subsetneq\dots$$
We have $r_1\in A$ for some domain  $A\in\mathbf S$. Let $\varphi=\varphi_A$.
Since $r_1\ne0$, there is an increasing sequence of nonzero  principal  ideals in the ring $A$:
$$r_1A= \varphi(r_1)A\sub \varphi(r_2)A\sub  \varphi (r_3)A\sub\dots$$

For each $n\ge1$, we have $\frac {r_n}{r_{n+1}}\in R\sm\U(R)$; hence
$\varphi\left(\frac {r_n}{r_{n+1}}\right)=\frac {\varphi(r_n)}{\varphi(r_{n+1})}\in A\sm \U(A)$.
It follows that all the inclusions in the sequence 
$$\varphi(r_1)A\sub\varphi(r_2)A\sub \varphi (r_3)A\sub\dots$$ are strict, contradicting the assumption that $A$ is Noetherian and so $R$ satisfies accp.

Assume for  $f\in R\setminus\{0\}$ and $g\in\Frac(R)$  that $fg^n\in R$ for all $n\ge1$.
Since the union of the subrings in $\mathbf S$ is directed, there exists a domain $A\in\mathbf S$ such that $f\in A$ and $g\in \Frac(A)$. Hence $fg^n\in R\cap \Frac(A)$, for all $n\ge1$. Since  $A$ is a retract of $R$, we have $A=R\cap \Frac(A)$. The domain $A$ is c.i.c. since $A$ is Noetherian and integrally closed. Hence  $g\in A\sub R$. Thus $R$ is c.i.c..
\end{proof}

\begin{lemma}\label{intclos}
	Let $A$ be a integrally closed domain, let $n\ge1$ and let $X,Y, Z_i \  (1\le i\le n)$ be independent indeterminates over $A$.
	Then the domain
	$$A[X,Y,Z_i, \frac {XZ_i}{Y^i}  \ (1\le i\le n)]$$
	is  integrally closed.
\end{lemma}

\begin{proof}
		Let $S$ be the multiplicative monoid generated by $X,Y,Z_i, \frac {XZ_i}{Y^i}$ 
	\mbox{$(1\le i\le n)$}. 
	We show that the  monoid $S$ is integrally closed. Let $G$ be the group of fractions of $S$, that is, $G$ is the multiplicative group  generated by $X,Y, Z_i\ (1~\le~i~\le~n)$. Let $g$ be an element of $G$ such that $g^k\in S$ for some integer $k\ge1$. Since the  monoid generated by $X,Y, \frac 1Y, Z_i \ (1\le i\le n) $ is integrally closed, it follows that $g$ belongs to this monoid. Thus
	$$
	g=X^fY^m\prod_{i=1}^nZ_i^{r_i},
	$$
	where $f, r_i$ are nonnegative integers for all $i$, and $m$ is an integer.   
	We have 
	\begin{equation}\label{eqintclos}
	g^k=X^{kf}Y^{km}\prod_{i=1}^nZ_i^{kr_i}=X^aY^b\prod_{i=1}^n Z_i^{c_i}\prod_{i=1}^n\left(\frac{XZ_i}{Y^i}\right)^{e_i},
	\end{equation}	
	where $a, b,c_i,e_i$ are nonnegative integers for all  $i$.
	We  may assume that the sum $a+\sum_{i=1}^n ic_i$ is minimal. 
	
	First assume that $c_i=0$ for all $i$.
	Comparing exponents of the indeterminates $Z_i$ on the two sides of \eqref{eqintclos}, we obtain that $e_i=kr_i$ for all $i$, so $a$ and $b$ are divisible by $k$. It follows that $g\in S$.
	
	Now assume that $c_{i_0}>0$ for some index $i_0$. If $a>0$, then
	$$g^k=X^{a-1}Y^{b+i_0}\left(Z_{i_0}^{c_{i_0-1}}\prod_{i\ne i_0} Z_i^{c_i}\right)\left(\frac{XZ_{i_0}}{Y^{i_0}}\prod_{i=1}^n\frac{X^{e_i}Z_i^{e_i}}{Y^{ie_i}}\right),$$ contradicting the minimality of $a+\sum_{i=1}^n ic_i$. Thus $a=0$.
	
	Let $j,q$ be integers such that $c_j>0$ and $e_q>0$. If $j>q$, we interchange $Z_j$ and $Z_q$ as follows:	
	$$
	g^k=Y^{b+j-q}\left(Z_qZ_j^{c_j-1}\prod_{i: i\ne j}Z_i^{c_i}\right)\left(\frac{XZ_j}{Y^j}\left(\frac{XZ_q}{Y^q}\right)^{e_q-1}\prod_{i:i\ne q}^n\left(\frac{XZ_i}{Y^i}\right)^{e_i}\right),
	$$
	contradicting the minimality assumption on $a+\sum_{i=1}^n ic_i$. Hence $j\le q$ for all $j$ and $q$ such that $c_j$ and $e_q$ do not vanish. 	
We have
	\begin{equation}\label{eqint}
	g^k=X^{kf}Y^{km}\prod_{i=1}^nZ_i^{kr_i}=Y^b\prod_{i=1}^n Z_i^{c_i}\prod_{i=1}^n\left(\frac{XZ_i}{Y^i}\right)^{e_i}.
	\end{equation}
	 Let $1\le q\le n$ be an integer such that $q\ne q_0=\min_{e_i>0 }i$. Since either $c_q=0$ or $e_q=0$,  and since by \eqref{eqint} we have $c_q+e_q=kr_q$, it follows that  both $c_q$ and $e_q$ are divisible by $k$. Comparing the exponents of $X$ on both sides
	of \eqref{eqint}, since all $ e_i$ for $i\ne q_0$ are divisible by $k$, we see that also $e_{q_0}$ is divisible by $k$. Clearly, also $c_{q_0}$ and $b$ are divisible by $k$. Thus $g\in S$, so 
 	the monoid $S$ is integrally closed. By \cite[Corollary 12.11 (2)]{gs}, the domain $D$ is integrally closed.
\end{proof}

\begin{remark} 
	The domain $D$ in Lemma \ref{intclos} is isomorphic to a subring of a polynomial ring over the domain $A$ in $n+2$ indeterminates. Indeed, for $U_i=\frac {Z_i}{Y^i} \ (0\le i\le n)$ we have $D=A[X,Y, XU_i, Y^iU_i \ (1\le i\le n)]\subseteq k[X,Y, U_i \ (0\le i\le n)]$. Similarly,  the domains $D$ of Example \ref{ex1} and  $A$ of Example \ref{ex2} below may be viewed as subrings of a polynomial ring over $k$ in infinitely many indeterminates.	 
\end{remark}	

\begin{example}\label{ex1} 
	A  completely integrally closed local domain $R$ with accp such that $R_P$ is not Archimedean for some prime ideal $P$.
\end{example}

Let $k$ be a field and let
$$D=k[X,Y, Z_n, \frac {XZ_n}{Y^n} \ (n\ge1)],$$
where $X,Y,Z_n \ (n\ge1)$ are independent indeterminates over $k$.
Let $M$ be the maximal ideal of $D$ generated by the elements \mbox{$X,Y, Z_n, \frac {XZ_n}{Y^n}\ (n\ge1)$}. Set
$$R=D_M  \text{ and } P=\langle X,Y, \frac {XZ_n}{Y^n} \ (n\ge1)\rangle R.$$

For each $n\ge1$, let $D_n=k[X,Y, Z_i, \frac {XZ_i}{Y^i}\ (1\le i\le n)]$ and $R_n=(D_n)_{M_n}$,
where $M_n$ is the maximal ideal of $D_n$ generated by $X,Y, Z_i, \frac {XZ_i}{Y^i}\ (1\le i\le n)$, thus $M_n=M\cap D_n$.

Clearly $R_1\sub R_2\sub\dots$ and $R=\bigcup_n R_n$.
For each $n$, there exists a retraction $\varphi_n: R\to R_n$ that maps to $0$ each indeterminate $Z_i$, for $i>n$. Clearly $\varphi_n(MR)\sub M_nR_n$. By Lemma \ref{intclos},  the domains $R_n$ are integrally closed. Since the domains $R_n$ are Noetherian, from Corollary \ref{noeth} it follows that  $R$ is c.i.c. and $R$ satisfies accp.

The ideal $P$ is prime since $P$ is the set of all rational functions in $R$ vanishing when plugging   in first $X=0$, and then $Y=0$ (all rational functions in $R$ are defined for $X=0$, and {\em after} plugging in $X=0$, we obtain a function defined for $Y=0$). For all $n\ge1$, the elements $Z_n$ are invertible in  $R_P$, so $\frac X{Y^n}\in R_P$. Since  $Y$ is not invertible in $R_P$, we see
that  the domain $R_P$ is not Archimedean.
\qed

\begin{example}
	\label{ex2}
	A completely integrally closed  domain $R$ satisfying accp with just two maximal ideals such that, for each maximal ideal $M$, the domain $R_M$ is not Archimedean.
\end{example}

Let $k$ be a field and let
$$A=k[X_1,Y_1, Z_{1,n}, \frac {X_1Z_{1,n}}{Y_1^n}; \ X_2,Y_2, Z_{2,n}, \frac {X_2Z_{2,n}}{Y_2^n}  \ (n\ge1)],$$
where $X_i,Y_i,Z_{i,n} (i=1,2,\ n\ge1)$ are independent indeterminates over $k$.

Let
\begin{align*}
P_1&=\langle X_1,Y_1, \frac {X_1Z_{1,n}}{Y_1^n}, Z_{2,n}, \frac {X_2Z_{2,n}}{Y_2^n} \ (n\ge1)\rangle A, \text{ and } \\
P_2&=\langle X_2,Y_2, \frac {X_2Z_{2,n}}{Y_2^n}, Z_{1,n}, \frac {X_1Z_{1,n}}{Y_1^n}\ (n\ge1)\rangle A.
\end{align*}
The ideal $P_1$ is prime since it is the set of all rational functions in $A$ vanishing when plugging   in first $X_1=Z_{2,n}=0$ for all $n$, and then $Y_1=0$. Similarly, the ideal $P_2$ is prime.

For all $n\ge1$, the elements $Z_{1,n}$ are invertible in  $A_{P_1}$, so $\frac {X_1}{Y_1^n}\in A_{P_1}$. Since  $Y_1$ is not invertible in $A_{P_1}$, we see
that  the domain $A_{P_1}$ is not Archimedean. Similarly, the domain $A_{P_2}$ is not Archimedean.

Let $S=A\sm (P_1\cup P_2)$, and
$R=A_S$, thus $R=A_{P_1}\cap A_{P_2}.$
Hence $R$ has just two maximal ideals, namely $M_1=P_1A_{P_1}\cap R$ and $M_2=P_2A_{P_2}\cap R$.
We have $R_{M_i}=A_{P_i}$ for $i=1,2$, so the domains $R_{M_1}$ and $R_{M_2}$ are not Archimedean.

For each $n\ge1$, let
$$A_n=k[X_1,Y_1, Z_{1,j}, \frac {X_1 Z_{1,j}}{Y_1^j};\ X_2,Y_2, Z_{2,j}, \frac {X_2 Z_{2,j}}{Y_2^j}   \ (1\le j\le n)]$$
and $R_n=(A_n)_{S_n}$,
where $S_n=S\cap A_n$.

By  Lemma \ref{intclos},  the domains
$$D_n=k[X_1,Y_1, Z_{1,j}, \frac {Z_{1,j}X_1}{Y_1^j}   \ (1\le j\le n)]$$ and $A_n=D_n[X_2,Y_2, Z_{2,j}, \frac {X_2Z_{2,j}}{Y_2^j}  \ (1\le j\le n)]$
are integrally closed. Hence $R_n$ is integrally closed.

Clearly $R_1\sub R_2\sub\dots$ and $R=\bigcup_n R_n$.  For each $n\ge1$ we have a retraction $\varphi_n: R=A_S\to R_n$ that maps to $0$ each indeterminate $Z_{i,j}$ for $i=1,2$ and $j>n$ since the elements $Z_{1,j}$ and $Z_{2,j}$ do not belong to $S$. Clearly the elements in $\varphi_n(M_1\cup M_2)$ are nonunits in $R_n$.
Since the domains $R_n$ are Noetherian and integrally closed, it follows from Proposition \ref{noeth} that $R$ is c.i.c. and $R$ satisfies accp.
\qed

\medskip
Of course, the domain $R$ in Example \ref{ex2} is not Mori, since any localization of a Mori domain is Mori and so Archimedean.

\begin{remark}\label{double}
	If $D$ and $A$ are the domains defined in Examples \ref{ex1} and   \ref{ex2}, respectively, then $A\cong D\otimes_k D$.
\end{remark}

The next example \ref{ex3} shows that a stable Archimedean domain may not be  locally Archimedean. We will use below the following well-known facts:

\begin{lemma}\label{val} (see \cite[Lemma 1.1.4 and Proposition 5.3.3]{fhp})
Let $U$ be a valuation domain (possibly a field), let $K=\Frac (U)$, and let $X$ be an indeterminate over $U$. Then \mbox{$V=U+XK[X]_{\langle X \rangle}$} is a valuation domain. If $U$ is strongly discrete, then also $V$ is strongly discrete.
The  prime ideals of $V$ are all the ideals $P+XK[X]_{{\langle X \rangle}}$, where $P$ is a prime ideal of $U$. Moreover, if  $P$ is nonzero, then  $P+XK[X]_{{\langle X \rangle}}=PV$ and $(P+XK[X]_{{\langle X \rangle}})\cap U=P$. For $P=(0)$ the ideal
$XK[X]_{{\langle X \rangle}}$  is the least nonzero prime ideal of $V$.  Thus, if $U$ is finite dimensional, then $\dim V=\dim U+1$.
\end{lemma}

\begin{corollary}\label{corval}
Let $X$ and $Y$ be two independent indeterminates over a field $k$,  let $C=k[Y,\frac X{Y^n}\ (n\ge1)]$, and let $P$ be the maximal ideal $YC=\langle X,\frac X{Y^n}\ (n\ge1)\rangle$ of $C$. Then
$V=C_P$ is a strongly discrete $2$-dimensional valuation domain.
\end{corollary}

\begin{proof}
Clearly, $V=k[Y]_{\langle Y\rangle}+Xk(Y)[X]_{\langle X \rangle}$. By Lemma \ref{val}, $V$ is a  strongly valuation domain of dimension $2$.
\end{proof}

\begin{example}\label{ex3}
A stable $2$-dimensional Pr\"ufer domain $R$ satisfying accp with just two maximal ideals of height $2$. Thus for each maximal ideal $M$ of $R$, except the two maximal ideals of height $2$, the domain $R_M$ is a DVR. Also $R$ is  Archimedean, but not locally Archimedean: $R_M$ is not Archimedean if $M$ is a maximal ideal of $R$ of height $2$.
\end{example}

Let $X$ and $Y$ be two independent indeterminates over a field $k$. Set
$$R=k[X,Y,  \frac{X(1-X)^n}{Y^n}, \  \frac{Y^{n+1}}{(1-X)^n}\ (n\ge1)]_S,$$
where $S=k[Y]\sm Yk[Y]$.

\smallskip
Let $T=\frac{1-X}Y$. We have $X=1-YT$, so
$$R=k[Y, YT, (1-YT)T^n, \frac Y{T^n}\ (n\ge1)]_S.$$
(As shown in item (1) below, $R$ satisfies accp. Hence $R$ is Archimedean, although $\frac Y{T^n}\in R$ for all $n\ge1$. This is not a contradiction since $T\notin R$.)

\smallskip
\begin{enumerate}
\item
\textbf{\textit{$R$ satisfies accp.}}
\begin{proof} 
Let $f$ and $g_n\ (n\ge1)$ be nonzero elements of  $R$ such that $\frac f{\prod_{i=1}^ng_i}\in R$ for all $n\ge1$. To prove that $g_i$ is a unit for $i\gg0$, we may assume that \mbox{$g_i\in k[Y, YT, (1-YT)T^n, \frac Y{T^n} \ (n\ge1)]$}, for all $i\ge1$.
t over $k$, we may view the ring \mbox{$k[Y, YT, (1-YT)T^n, \frac Y{T^n}\ (n\ge1)]$} as a subring of the polynomial ring $k(T)[Y]$. Thus for $i\gg0$ we have $\deg_Y (g_i)=0$, that is, $g_i\in k(T)$ .
Since the elements $Y$ and $T$ are algebraically independen

For  $i\gg0$, since $g_i\in k[Y, YT, (1-YT)T^n, \frac Y{T^n} \ (n\ge1)]$,
by plugging in $Y=0$, we obtain that $g_i\in k[T]$;  by plugging in $Y=\frac 1 T$, we obtain that $g_i\in k[\frac 1T]$, so $g_i\in k[T]\cap k[\frac 1 T]=k$.
We conclude that $R$ satisfies accp.	 
\end{proof} 	  

\item
\textbf{\textit{$R$ is a  stable $2$-dimensional Pr\"ufer domain with just two maximal ideals of height $2$.} }

\begin{proof} 
Let $M$ be  a maximal ideal of $R$.

\emph{Assume that $Y\in M$. We will prove that $R_M$ is a stable $2$-dimensional  valuation domain, in particular $\height M=2$.}{}

Since $X(1-X)\in RY\sub M$ it follows that either $X\in M$ or $1-X\in M$.

	\emph{First assume that  $Y,X\in M$.}
		
	Clearly,   $R\subseteq C=k[Y,\frac X{Y^n}\ (n\ge1)]\sub R_M$. Hence $MR_M=C_P$ for a unique prime ideal $P$ of $C$. Since the maximal ideal $\langle Y,\frac X{Y^n}\ (n\ge1)\rangle$ of $C$ is contained in $MR_M$, it follows that this ideal is equal to $P$.  By Corollary \ref{corval}, $R_M=C_P$ is a $2$-dimensional strongly discrete, and so stable,  valuation domain, and $\height M=2$.

\emph{Next assume that $Y,1-X\in M$}.	

	Recall that $T=\frac {1-X}Y$. Since $XT=\frac{X(1-X)}Y\in R$ and $X$ is a unit in $R_M$, we see that $T\in R_M$. Hence 
	$$R\subseteq k[T,\frac Y{T^n}\ (n\ge1)]\sub R_M$$ 
	Since $T$ and $Y$ are algebraically independent over $k$, we conclude as before that there is  unique maximal ideal $M$ of $R$ containing $Y$ and $1-X$, and that $R_M$ is a $2$-dimensional stable valuation domain.	

\emph{Now assume that $Y\notin M$}. 

Let $D=k(Y)[X]$. Thus $R\subseteq D\subseteq R_M$, so $R_M=D_P$ for a unique prime ideal $P$ of $D$. Since $D$ is a  principal ideal domain, it follows that $R_M$ is a DVR. 
Moreover, every nonunit in  $R$ belongs to just finitely many domains of the form $D_P$ for a prime ideal $P$ of $R$ since $R$	is a PID. 

Since $R$ has just  $2$ maximal ideals of height $>1$, it follows 
that $R$ has finite character.
In addition,  each localization of $R$ at a maximal ideal is a stable valuation domain, implying that $R$ is a stable Pr\"ufer domain \cite[Theorem 3.3]{O2}. 
\end{proof}
\end{enumerate}

Thus  $R$ is $2$-dimensional with exactly $2$ maximal ideals of height $2$.  The localizations at these two maximal ideals are not Archimedean, as seen directly from the above proof. Actually, as it is well-known, a valuation domain is Archimedean if and only if it is one-dimensional. The maximal ideals of $R$ are invertible since $R$ is stable and Pr\"ufer. Thus in  Corollary \ref{archPID} (2) we may not assume just that $R$ is Archimedean rather than locally Archimedean.

Example \ref{ex3} shows that the converse of Corollary \ref{stableaccp} is false: a stable domain $R$ which satisfies accp need not be locally Archimedean, even if $R$ is  Pr\"ufer and $2$-dimensional.

\begin{example}\label{ex4}
	A local integral domain $(R,M)$ with the following properties:
	\begin{enumerate} 
		\item 
		$R$ is one-dimensional, Noetherian, not (finitely) stable, but with stable maximal ideal.
		
		\item
		$R'=(M:M)$ is a finitely generated $R$-module.
		
		\item
		$R'$ is a principal ideal local domain, so $R'$ is stable and Pr\"ufer.
	\end{enumerate}
\end{example}

Let $K=\mathbb Q(\sqrt[3] 2)$ and $R=\mathbb Q+XK[[X]]$. Thus $R'=K[[X]]$ is a principal ideal local domain with maximal ideal $M=XK[[X]]$; so $R$ is one-dimensional, and $R'$ is a $3$-generated $R$-module. By the Eakin-Nagata Theorem, $R$ is Noetherian.
Clearly, $R'$ is not a quadratic extension of $R$, so $R$ is not finitely stable. Explicitly, the fractional ideal $I=\langle 1,\sqrt[3]2\rangle$ of $R$ is not stable (equivalently, the ideal $\langle X, X\sqrt[3]2\rangle$ of $R$ is not stable). Indeed, $I^2=\langle 1,\sqrt[3]2,\sqrt[3]4\rangle$ and   $(I:I^2)=XR$, so $I(I:I^2)= XR\ne R$. It follows that $I$ is not stable. The maximal ideal $M$ of $R$ is stable, since $M$ is an ideal of the stable domain $R'$ which is an overring of $R$.\qed

\medskip

In the next example we present a well-known construction which is related to the  construction in  the proof of the Kaplansky-Jaffard-Ohm Theorem \cite[Ch.III, Theorem 5.3]{FS}. This example illustrates explicitly a particular case of Olberding's Theorem \cite[Proposition 5.4]{O3}, and will be also used for Example \ref{ex6}. 

\begin{example}\label{ex5}
For each $1\le n\le \infty$ and for a field $k$, a strongly discrete, so stable, $n$-dimensional valuation domain $V$ containing  $k$. In particular, if $n=\infty$,  the nonzero prime ideals of $V$ form a  descending infinite sequence, so the height of every nonzero prime ideal of $V$ is infinite. Moreover, for all $n$, $\Frac V$ is a purely transcendental extension of $k$ of transcendence degree $\aleph_0$.
\end{example}

First let $n=\infty$.
Let $V=A_Q$, where
$$A=k[X_n,\frac {X_{n+1}}{X_n^i}\ (n\ge1, i\ge1)],$$
$k$ is a field, $X_n\ (n\ge1)$ are independent indeterminates over $k$, and $Q=X_1A=\langle X_n,\frac {X_{n+1}}{X_n^i} \ (n\ge1)\rangle A$ is a maximal ideal of $A$.

It is easy to show that $V=\bigcup_{n=1}^\infty V_n$ (an ascending union), where $V_n$ are  subrings of $V$ defined inductively as follows:
$V_0=k$, and for $n\ge1$, we let $V_n=V_{n-1}+X_n\left(k(X_1,\dots, X_{n-1})[X_{n}]\right)_{\langle X_n\rangle}$.

By induction, $\Frac(V_n)=k(X_1,\dots,X_n)$ for $n\ge1$. Thus $V_n=V_{n-1}+X_n\left(\Frac(V_{n-1})[X_n]_{\langle X_n\rangle}\right)$ for $n\ge1$. Hence by Corollary \ref{corval}, we obtain inductively that $V_n$ is a strongly discrete valuation domain of dimension $n$, with maximal ideal $M_n=X_1V_n$, and that   the nonzero prime ideals of $V_n$ form a descending chain
$$M_n=P_{n,n}\supsetneq P_{n,n-1} \supsetneq \dots \supsetneq P_{n,1}.$$
It follows that  the domain $V=\bigcup_{n=1}^\infty V_n$ is a strongly discrete, so stable,  valuation domain with maximal ideal $M=X_1V$.
Let  $P$ be a nonzero   prime ideal of $V$. Since
$P=\bigcup_{n=1}^\infty (P\cap V_n)$,
we have $P\cap V_n\ne (0)$ for some integer $n\ge1$. By Lemma \ref{val}, $P=(P\cap V_n)V=P_{n,i}V$ for an integer $1\le i\le n$.
If $n$ is minimal, then $P_{n,i}$ is the least nonzero prime ideal of $V_n$, so $i=1$.
Hence
the nonzero prime ideals of $V$
form an infinite descending chain $M=P_1\supsetneq P_2 \supsetneq \dots $,
where $P_n=P_{n,1}V$ for all  $n\ge1$.

Thus for all $n\ge1$, $P_n$ is the ideal of $V$ generated by the one-dimensional subspace $X_n k(X_1,X_2,\dots,X_{n-1})$ of $V$ over the field $k(X_1,X_2,\dots,X_{n-1})$.

Explicitly, for all $n\ge1$  we have
$$P_n=\sum_{i=n}^\infty X_i\left(k(X_1,\dots,X_{i-1})[X_i]_{\langle X_i\rangle}\right).$$

If $n$ is finite, similarly to the definition of $V$ above,  we define 
$V_n=A_{Q_n}$, where
$$A=k[X_j,\frac {X_{j+1}}{X_j^i}\ (1\le1<n, i\ge1)],$$
 and $Q=X_1A$ is a maximal ideal of $A$.
(if $n=1$, then $A=k[X_1]$).
\qed

\medskip

In the last example we exhibit an $n$-dimensional Archimedean stable local domain, for each $n\geq 2$; thus answering in the negative the question posed in \cite[Problem 7.1]{G}. 
(For details concerning this example, see Propositions \ref{V1V2} and \ref{intV1V2} above.) 

We need the following lemma:

\begin{lemma}\label{DVRLk} 
	Let $k$ be a field, and let $L\ne k$ be a purely transcendental field extension of $k$ with $\trd L/k\le \aleph_0$. Then there exists a DVR $(V,N)$ such that $\Frac V=L$ and $V/N=k$.	 
\end{lemma}

\begin{proof}
	Let $L=k(B)$, where $B$ is a set of algebraically independent elements over $k$. Since $\trd L/k\le \aleph_0\le \trd k((X))/k$ \cite[Lemma 1, Section 3]{ML}, there exists a subset $B_0$ of $k((X))$ containing $X$ such that $|B_0|=|B|$. Thus there exists an isomorphism over $k$ of the  fields $L$ and $k(B_0)$ mapping $B$ onto $B_0$. Hence we may assume that $L=k(B)\subseteq k((X))$ and that $X\in B$. 
	Define
	$
	V=k[[X]]\cap L.
	$
	Thus $V$ is a DVR with maximal ideal $XV$, and $V/XV\cong k$. Since $k[B] \subseteq V \subseteq  L=k(B)$, it follows that $\Frac(V)=L$. 
\end{proof}

\begin{example}\label{ex6}
For $1\le n\le\infty$, a stable $n$-dimensional Archimedean local domain $(R,M)$. \end{example}

By Example \ref{ex5}, for any field $k$, there exists a stable $n$-dimensional valuation domain $(V_2, Q_2)$ containing $k$ such that $\Frac V_2=L$ is a purely transcendental extension of $k$ and $V_2/Q_2=k$. By Lemma \ref{DVRLk}, there exists a DVR $(V_1,Q_1)$ containing $k$ such that $\Frac(V_1)=L$, and $V_1/Q_1=k$.
By Proposition \ref{V1V2}, there exists a local Archimedean finitely stable domain $R$ such that $R'=V_1\cap V_2$ and by Proposition \ref{intV1V2} such a domain is stable.\qed

\medskip

By Example \ref{ex6} and by Proposition \ref{intV1V2} (1), the integral closure of an Archimedean domain, or even an accp stable domain, is not necessarily Archimedean.  The domain $\mathbb Z+X\ol{\mathbb Z}[X]$, where $\ol{\mathbb Z}$ is the ring of all algebraic integers, satisfies accp
while $R'=\ol{\mathbb Z}[X]$ does not, although $R'$ is Archimedean \cite[Example 5.1]{AAZ1}.

\setcounter{npart}0
\renewcommand{\comm}{}
\newpage
\renewcommand\refname{References \lowercase{(for both parts)}} 

\end{document}